\documentclass[a4paper,11pt]{article}
\usepackage[latin1]{inputenc}
\usepackage[T1]{fontenc}
\usepackage[english]{babel}

\usepackage{charter}

\usepackage{indentfirst}		
\usepackage[dvipsnames]{xcolor}			

\usepackage[nodayofweek]{datetime}

\usepackage{hyperref}
\hypersetup{colorlinks,linkcolor=red,citecolor=blue,urlcolor=green}

\usepackage[nobysame
           ,alphabetic
           ,initials
           ]{amsrefs}

\usepackage{pifont}
\usepackage{graphicx}
\usepackage{subfigure}

\usepackage[top=2.5cm, bottom=3.5cm, left=2.35cm, right=2.35cm]{geometry}

\usepackage{amsmath}
\usepackage{amsthm}	
\usepackage{amsfonts}
\usepackage{amssymb,bbm}						

\newtheorem{definition}{Definition}[section]				
\newtheorem{theorem}[definition]{Theorem}				
\newtheorem{proposition}[definition]{Proposition}
\newtheorem{corollary}[definition]{Corollary}
\newtheorem{lemma}[definition]{Lemma}

\theoremstyle{definition}											
\newtheorem{remark}[definition]{Remark}

\numberwithin{equation}{section}
\numberwithin{figure}{section}

\usepackage{
           ulem 
} 

\definecolor{darkgreen}{rgb}{0,0.35,0}

\newcommand{\N}{\mathbb{N}}																				
\newcommand{\Z}{\mathbb{Z}}																					
\newcommand{\R}{\mathbb{R}}																					
\newcommand{\F}{\mathcal{F}}																					
\newcommand{\bE}{\mathbb{E}}																				
\newcommand{\abs}[1]{\lvert#1\rvert }																		
\newcommand{\Abs}[1]{\left\lvert#1\right\rvert }														
\newcommand{\norm}[1]{\lVert#1\rVert}																	
\newcommand{\half}{\frac{1}{2}}																				
\DeclareMathOperator*{\trace}{\textrm{Tr}}															
\newcommand{\scalar}[2]{\left\langle #1 \big| #2 \right\rangle}								
\newcommand{\goesto}{\rightarrow}																			

\newcommand{\cI}{\mathcal{I}}

\newcommand{\cN}{\mathcal{N}}

\newcommand{\cP}{\mathcal{P}}

\newcommand{\cR}{\mathcal{R}}

\newcommand{\cT}{\mathcal{T}}

\newcommand{\cY}{\mathcal{Y}}
\newcommand{\cZ}{\mathcal{Z}}

\newcommand{\bR}{\mathbb{R}}

\newcommand{\ud}{\mathrm{d}}
\newcommand{\uds}{\mathrm{d}s}
\newcommand{\udt}{\mathrm{d}t}
\newcommand{\udu}{\mathrm{d}u}

\newcommand{\ti}{{t_{i}}}
\newcommand{\tip}{{t_{i+1}}}
\newcommand{\tj}{{t_{j}}}
\newcommand{\ip}{{i+1}}
\newcommand{\Hip}{H_{i+1}}
\newcommand{\wbHip}{\wb{H}_{i+1}}
\newcommand{\hip}{h_{i+1}}

\newcommand{\wh}[1]{\widehat{#1}}
\newcommand{\wt}[1]{\widetilde{#1}}
\newcommand{\wb}[1]{\overline{#1}}
\newcommand{\rh}[1]{\check{#1}}			

\newcommand{\bEix}{\bE^{\ti,x}} 

\newcommand{\hfmon}{(Mon) }
\newcommand{\hfreg}{(RegZ) }
\newcommand{\hfregY}{(RegY) }

\title{Full-Projection explicit FBSDE scheme for parabolic PDEs with superlinear nonlinearities}
\author{
\normalsize Arnaud Lionnet \\[8pt] 
        \small Ecole Normale Supérieure \\
        \small de Cachan\\  
        \small \\
        \small  first.last at ens-cachan.fr
 \and
 	\normalsize Gon\c calo dos Reis \footnote{G. dos Reis acknowledges support from the \emph{Funda{\c c}$\tilde{\text{a}}$o para a Ci$\hat{e}$ncia e a Tecnologia} (Portuguese Foundation for Science and Technology) through the project UID/MAT/00297/2013 (Centro de Matem\'atica e Aplica\c c$\tilde{\text{o}}$es CMA/FCT/UNL).} \\[8pt]
         \small  University of Edinburgh\\ 
         \small  and \\
	\small  Centro de Matem\'atica e Aplica\c c$\tilde{\text{o}}$es\\
        \small  G.dosReis@ed.ac.uk
 \and
         \normalsize Lukasz Szpruch \footnote{L. Szpruch acknowledges  support from The Alan Turing Institute under the EPSRC grant EP/N510129/1 }  \\[8pt]
         \small  University of Edinburgh\\
         \small  and \\
	     \small  The Alan Turing Institute \\ 
         \small  l.szpruch@ed.ac.uk
}

\date{ }

\begin{document}
\selectlanguage{english}
\maketitle

\vspace{-0.5cm}
\begin{abstract}

Developing efficient and stable approximations for high dimensional PDEs is of key importance for numerous applications. The language of Forward-Backward Stochastic Differential Equations (FBSDE), with its nonlinear Feynman-Kac formula, allows for purely probabilistic representations of the solution and its gradient for parabolic nonlinear PDEs.  

In this work we build on the recent results of \cite{LionnetReisSzpruch2015} by introducing and studying a Full-Projection explicit time-discretization scheme for the approximation of FBSDEs with non-globally Lipschitz drivers of polynomial growth. We establish convergence rates and we show that, unlike classical explicit schemes, it preserves stability properties present in the continuous-time dynamics, in particular, the scheme is able to preserve the possible coercivity/contraction property of the PDE's coefficients. 

The scheme is then coupled with a quantization-type approximation of the conditional expectations on a space-time grid in order to provide a complete approximation scheme for these FBSDEs/nonlinear PDEs and a full analysis is also carried out. We illustrate our findings with numerical examples.
\end{abstract}

polynomial growth driver, non-explosion.


\section{Introduction}

The theory of Forward-Backward Stochastic Differential Equations (FBSDE) paves a way for probabilistic numerical methods to approximate solutions to nonlinear parabolic PDEs which are used to describe many biological and physical phenomena. Indeed, the solution $v$ to the PDE 
	\begin{align*}
		\Big( \partial_t v + \half (\sigma \sigma^*) \cdot v_{xx} + b \cdot v_x  \Big)(t,x) = f(\cdot,v,v_x \sigma)(t,x) \
		\quad \text{with} \quad 
		v(T,x)=g(x),
	\end{align*}
is given by solving, for $(s,x_0) \in [0,T] \times \R^d$, for $t \in [s,T]$, the FBSDE system \eqref{equation---canonicSDE}-\eqref{equation---canonicBSDE} given by
	\begin{align}
		\label{equation---canonicSDE}
		X^{s,x_0}_t 
		&
		= x_0 + \int_s^t b(u,X^{s,x_0}_u)\uds + \int_s^t \sigma(u,X^{s,x_0}_u)\ud W_u,
		\\
		\label{equation---canonicBSDE}
		Y^{s,x_0}_t 
		&
		= g(X^{s,x_0}_T) + \int_t^T f(u,X^{s,x_0}_u,Y^{s,x_0}_u,Z^{s,x_0}_u) \udu - \int_t^T Z^{s,x_0}_u \ud W_u.
	\end{align}
and setting $v(t,X^{s,x_0}_t)=Y^{s,x_0}_t$ (in particular $v(s,x_0)=Y^{s,x_0}_s$) and $\nabla_x v(t,X^{s,x_0}_t)\sigma(t,X^{s,x_0}_t)=Z_t^{s,x_0}$ (see \cite{ElKarouiPengQuenez1997}), where $W$ is a Brownian motion. In other words, the theory of FBSDE provides the nonlinear extension to the celebrated Feynman-Kac stochastic representation formula for linear parabolic PDEs (i.e when $f=0$) and opens way to the use of the full Monte-Carlo machinery for such class of PDEs. 

The majority of the results on the numerical methods for FBSDEs relies on the global Lipschitz assumption on $f$, which is often not satisfied. This is the case for a number of important PDEs of reaction-diffusion type, like the Allen--Cahn equation, the FitzHugh--Nagumo equations, the Fisher--KPP equation or the standard nonlinear heat and Schr\"odinger equation (see \cite{Henry1981}, \cite{Rothe1984}  and references), where $f$ is a polynomial in $Y$.

\paragraph*{}
Since the first papers of \cites{Zhang2004,BouchardTouzi2004}, a dynamic literature has addressed the numerical approximation of \eqref{equation---canonicSDE}-\eqref{equation---canonicBSDE}. 
A number of works have been concerned with the time-discretization of the continuous dynamics, introducing Picard-type schemes \cite{BenderDenk2007}, 
second order methods \cite{CrisanManolarakis2012}, Runge--Kutta methods \cite{ChassagneuxCrisan2014} or linear multistep methods \cite{Chassagneux2014}.
A significant effort has also been directed to the methods for approximating the conditional expectations involved in the time-discretization schemes:  
regression on a basis of function \cite{GobetLemorWarin2006}, Malliavin weights \cite{GobetTurkedjiev2016}, chaos decomposition \cite{BriandLabart2014}, quantization \cite{SagnaPages2015}, cubature methods \cite{CrisanManolarakis2012} and spatial-grid approximations \cite{MaShenZhao2008}. 
More recently, \cite{ChassagneuxRichou2015} and \cite{Lionnet2016} have focused on the numerical stability (rather than the convergence) of the standard Bouchard--Touzi--Zhang (BTZ) schemes \cites{Zhang2004,BouchardTouzi2004}.
Given a time grid $\pi = (\ti)$ of modulus $\abs{\pi}=h=\max{\{h_i\}}$ where $h_\ip=\tip-\ti$, used to discretize the time interval $[0,T]$, and given $\theta \in \{0,1\}$,
these BTZ schemes (and more generally $\theta$-schemes with $\theta \in [0,1]$) can formally be described by the recursive application of 
the family of operators 
$\Phi_i^\theta : 	L^p(\F_\tip,\R^n)  \rightarrow L^p(\F_\ti,\R^n) \times L^p(\F_\ti,\R^{n \times d})$
	\begin{align*}
				 \Phi_i^\theta : \cY & \mapsto (Y,Z) = \Big(  \bE_i\big[ \cY + (1-\theta) f(\cY,Z) h \big] + \theta f(Y,Z) \hip  ,   \bE_i\big[  \cY (W_\tip-W_\ti)^* \hip^{-1} \big]  \Big) .
	\end{align*}	
The above-mentioned works mostly assume a Lipschitz driver for the BSDE.  
\cites{LionnetReisSzpruch2015,LionnetReisSzpruch2016} undertook the analysis of approximation schemes for BSDEs when the driver has polynomial growth in the $Y$ variable and satisfy a monotonicity condition. 
While the implicit ($\theta = 1$) and mostly implicit ($\theta \ge 1/2$) schemes were found to be converging, the explicit scheme ($\theta = 0$) can explode when the terminal condition $g$ is sufficiently big.

In general, implicit schemes have better properties than explicit ones, nonetheless, in the context of superlinear monotone growth one wants to avoid implicit schemes altogether as a (usually expensive) secondary implicit equation needs to be solved at each time step. Speedy explicit schemes, especially in Monte-Carlo contexts, have an added appeal for this framework. As a remedy to guarantee the convergence of the explicit schemes while maintaining its explicit nature and thereby its lower computational cost, \cite{LionnetReisSzpruch2016} proposed to replace the driver $f$ by a family of drivers $f^h$, which are of linear growth for each $h$ and relax to $f$ as $h \goesto 0$ (see \cites{HutzenthalerJentzenKloeden2012,ChassagneuxJacquierMihaylov2016} for the SDE context).

In this work we introduce a new converging fully explicit scheme under a general monotonicity assumption and super-linear growth, called \emph{Full-Projection Scheme} (see \eqref{equation---reference--FPS.scheme} below). Notably, we show that, even though explicit, the scheme enjoys \emph{numerical stability} properties that, in general, cannot be obtained by simply modifying $f$ into $f^h$ (see similar observation in \cite{SzpruchZhang2013} in the context of SDEs). In particular, the scheme is able to preserve the possible coercivity/contraction property of the PDE's coefficients. As a additional benefit, the truncation on both driver and process makes the stability analysis of the scheme much more natural.   

In the language of the $\Phi_i^\theta$ operator above, the Full-Projection Scheme (with $\theta=0$) is such that, at each time step, the input is truncated via a projection $T^h$ relaxing to the identity as $h \goesto 0$. The scheme can be described by the family of operators $(\Phi^{0,h}_i)$ with $\Phi^{0,h}_i = \Phi^{0}_i \circ T^h$.

\paragraph*{}
In section \ref{section---setting} we specify the assumptions, describe our time-discretization scheme and give the scheme's convergence rate. We follow with Section \ref{section---numerical.scheme.and.results} to illustrate the results with numerical simulations and the full error statement. The discretization error is analyzed in section \ref{section---analysis.of.the.time.discretization.scheme}, 
while the numerical error analysis for the full scheme's implementation is carried out in Section \ref{section---numerical.scheme.and.results}.

\section{Assumptions, scheme and main result}
\label{section---setting}

For the system \eqref{equation---canonicSDE}-\eqref{equation---canonicBSDE}, $T > 0$ is a fixed time-horizon and $W$ is a $d$-dimensional Brownian motion on a filtered probability space $(\Omega, (\F_t)_{t \in [0,T]}, P)$ where $(\F_t)$ is the usual augmentation of the natural filtration of $W$.

\subsection{Assumptions}
\label{sec:assumptions}

For the coefficient functions in system \eqref{equation---canonicSDE}-\eqref{equation---canonicBSDE} we assume the following conditions hold throughout.
Let $x_0\in\bR^d$; the functions $b : [0,T] \times \R^d \rightarrow \R^d$ and $\sigma : [0,T] \times \R^d \rightarrow \R^{d \times d}$ are the drift and diffusion coefficient of the forward process $X$, and are assumed to be $\half$-H\"older in time and Lipschitz in space;  
$g : \R^d \rightarrow \R^n$ is the terminal condition, assumed to be Lipschitz.

As for the driver function $f : \R^n \times \R^{n \times d}$ of \eqref{equation---canonicBSDE}, it would be possible to consider a driver that also depends on the time $t$ and the position $X_t$ of the forward process, as in \cite{LionnetReisSzpruch2015}, although it is tedious and does not change the global picture. We thus focus on a function of $Y$ and $Z$, for clarity. 
We assume that $f$ is monotone with constant $M_y \in \R$, is locally Lipschitz in $y$ with polynomially growing local constant induced by $L_y \ge 0$ and $m \in \N^*$, and is $L_z$-Lipschitz in $z$, $L_z \ge 0$. So for all $y,y',z,z'$ we have
	\begin{itemize}
		\item[] \hfmon : $\scalar{y'-y}{f(y',z)-f(y,z)} \le M_y \abs{y'-y}^2$, 
		\item[] \hfregY : $\Abs{ f(y',z)-f(y,z) } \le L_y \big( 1 + \abs{y'}^{m-1} + \abs{y}^{m-1} \big) \abs{y'-y}$,
		\item[] \hfreg : $\Abs{ f(y,z')-f(y,z) } \le L_z \abs{z'-z}$.
	\end{itemize}
Note that the case $m=1$ corresponds to a Lipschitz driver, so we generally assume $m \ge 2$.
It follows from these regularity assumptions that $f$ has polynomial growth in $y$ and $z$ of degrees $m$ and $1$ respectively, and has monotone growth.
So for all $y,z$, we can write
	\begin{align} \label{equation---assumptions.growth}
	\begin{aligned}
		\abs{f(y,z)} & \le K + K_y \abs{y}^m + K_z \abs{z}							\qquad \text{and}	\qquad 		
		\scalar{y}{f(y,z)} & \le M + \wh {M}_y \abs{y}^2 + M_z \abs{z}^2.
	\end{aligned}
	\end{align}
The constants in these bounds can be taken to be $K=\abs{f(0,0)}+L_y,K_y=2L_y,K_z=L_z$, while $M=\frac{\abs{f(0,0)}^2}{2\nu}$, $\wh M_y = M_y+\nu$ and $M_z=\frac{L_z^2}{2 \nu}$, for any $\nu > 0$. We sometimes make use of \eqref{equation---assumptions.growth} for transparency of the property being used.

\subsection{The Full-Projection Scheme}
\label{sec:fullprojectionscheme}

Let $\pi = (t_i)_{i = 0, \ldots, N}$ be the uniform time-grid (or subdivision) of the interval $[0,T]$ with $N+1$ points $\ti = ih$ for $i \in \{0, \ldots, N\}$ and step size $h = \tip - \ti = T/N = \abs{\pi}$, the modulus of the time-grid. Since $N \ge 1$, we have $h \le T$.
We restrict ourselves to uniform partitions for notational simplicity, although the results would hold for more general partitions.

For the forward process $X$ in \eqref{equation---canonicSDE}, we use the explicit Euler scheme yielding the family $(X_i)_i$ for $i \in \{0,\ldots,N\}$. 
So we define $X_0 = x_0$ and for all $i \in \{0,\ldots,N-1\}$, 
	\begin{align} 	\label{equation---reference--X.Euler.scheme}
		X_\ip = X_i + b(\ti,X_i) h + \sigma(\ti,X_i) \Delta W_\tip,\qquad \text{where} \quad \Delta W_\tip = W_\tip - W_\ti.
	\end{align}
For the backward processes, we consider a modified explicit scheme, called \emph{Full-Projection scheme}. It is initialized with $Y_N = \xi^N = g(X_N)$ and then, for $i$ decreasing from $N-1$ to $0$, we set
	\begin{align}			
	\label{equation---reference--FPS.scheme}
	Y_i = \bE_i\Big[ T^h(Y_\ip) + f\big(T^h(Y_\ip),Z_i\big) h \Big]
	\quad\text{and}\quad
			Z_i = \bE_i\big[  T^h(Y_\ip) \Hip^* \big] .
	\end{align}
In the equation above, $\bE_i := \bE[\cdot|\F_\ti]$.
The map $T^h$ is a projection operator on the ball of radius $R^h = R_0 h^{-\alpha}$ in $\R^n$, for some constant $\alpha  > 0$.
Typically, it could be given by $T^h(y) = \min(1,R^h \abs{y}^{-1}) y$. 
The properties of $T^h$ that we will use are: $\abs{T^h(y)} \le \abs{y}$, $\abs{T^h(y)} \le R^h$ for all $y \in \R^n$, $T^h(y)=y$ when $\abs{y} \le R^h$,
and $T^h$ is $1$-Lipschitz. 
In section \ref{section---numerical.scheme.and.results} we use a mollification of it, which coincide with the identity inside the ball of radius $R^h$, 
coincide with the projection outside the ball of radius $R_\epsilon^h = R_0 h^{-\alpha}+\epsilon$ for some small $\epsilon$, while remaining $1$-Lipschitz. 
The weight $\Hip$ is a $\R^{d \times 1}$-valued random variable that satisfies:
	\begin{enumerate}
		\item[a.] $\Hip$ is independent from $\F_\ti$ and is a martingale increment: $\bE_i[\Hip]=0$,
		\item[b.] 
			$\bE_i[\Hip \Hip^*] = \Lambda h^{-1} I_d$, so $\bE_i[\abs{\Hip}^2] = \Lambda d h^{-1}$, where $\Lambda = \Lambda(h) \le 1$ and $I_d = \text{Identity}_{\R^d}$,
		\item[c.] There exists a positive constant $C$ such that
			\begin{align}
			\label{equation---Error.of.approximating.by.Hip}
				\max_{0 \le i \le N-1} \bE\bigg[ \Abs{ \frac{\Delta W_\tip}{h} - \Hip }^2 \bigg] \le C h.
			\end{align}
	\end{enumerate}
A typical choice is naturally $\Hip = {\Delta W_\tip}/{h}$. This guarantees the convergence of the scheme. However, in order to have numerical stability, it is required to truncate the Brownian increment, and so we take $\Hip = {T^{r^h}(\Delta W_\tip)}/{h}$, where $T^{r^h}$ is a truncation as above, on the ball of radius $r^h = \sqrt{2h} \ln(1/h)$ in $\R^d$ (see, e.g., \cites{Lionnet2016,LionnetReisSzpruch2016}, Appendix).

Note that, at each step, the scheme above is the composition of the explicit BTZ scheme with the truncation $T^h$. It may be useful to keep track of a slightly different object, and define $(\wt Y_i,\wt Z_i)$ as follows: $\wt Y_N = T^h(g(X_N))$ and for $i$ decreasing from $N-1$ to $0$,
	\begin{align}		
	\label{equation---reference--FPS.scheme.alternative}
			\wt Y_i = T^h\Big( \bE_i\big[ \wt Y_\ip + f(\wt Y_\ip, \wt Z_i) h \big] \Big)
			\quad\textrm{and}\quad
			\wt Z_i = \bE_i\big[  \wt Y_\ip \Hip^* \big]
	\end{align}
At times is useful to think of the schemes as a composition of maps ---the steps of the schemes. 
Let us first describe the explicit ($\theta=0$) BTZ scheme for which one step is given by map $\Phi_i^0 = (\Phi_i^{0,Y},\Phi_i^{0,Z})$, defined as 
$\Phi_i^0 : L^p(\R^n,\F_\tip) \rightarrow L^p(\R^n) \times L^p(\R^{n \times d})$
	\begin{align*}
		\Phi_i^0:
				\cY & \mapsto (Y,Z) = \Big(  \bE_i\big[ \cY + f(\cY,Z) h \big] \ , \  \bE_i\big[  \cY \Hip^* \big]  \Big) .
	\end{align*}	
A step of the scheme \eqref{equation---reference--FPS.scheme} is then $\Phi_i := \Phi^0_i \circ T^h$, and the family of random variables $(Y_i,Z_i)_i$ is produced recursively by $(Y_i,Z_i)=\Phi_i(Y_\ip)$. A step of \eqref{equation---reference--FPS.scheme.alternative} is $(T^h \circ \Phi^{0,Y}_i,\Phi^{0,Z}_i)$.

Using this formalism, one can easily obtain the following properties:
	\begin{enumerate}
		\item $(Y_i,Z_i) = \Phi^0_i (\wt Y_\ip) $ for all $i \in \{0,\ldots,N-1\}$,
		\item $\wt Y_i = T^h (Y_i) $ for all $i \in \{0,\ldots,N\}$,
		\item $\wt Z_i = \Phi^{0,Z}_i(\wt Y_\ip) = \Phi^{0,Z}_i(T^h (Y_\ip)) = Z_i$ for all $i \in \{0,\ldots,N-1\}$. 
	\end{enumerate}
The pre-truncation or post-truncations, i.e. schemes \eqref{equation---reference--FPS.scheme} or \eqref{equation---reference--FPS.scheme.alternative}, are thus essentially equivalent.

\subsection{Convergence of the Full-Projection Scheme}

As is usual for the error analysis of BSDE schemes, we study the squared total errors
	\begin{align*}
		\mathrm{ERR}^Y(\pi)^2 = \max_{i = 0, \ldots, N} \bE\big[ \abs{Y_\ti - Y_i}^2 \big]
		\qquad \text{and} \qquad 
		\mathrm{ERR}^Z(\pi)^2 = \sum_{i = 0}^{N-1} \bE\big[ \abs{\wb Z_\ti - Z_i}^2 \big] h,
	\end{align*}
where the random variables $\wb Z_\ti$ are given by $\wb Z_\ti= h^{-1} \bE_i[\int_\ti^\tip Z_t \udt]$.
The square error between the continuous-time solution $(Y,Z)_{t\in[0,T]}$ and the discrete family $(Y_\ti,\wb Z_\ti)_{\ti\in\pi}$ is well understood, following from the path-regularity theorem (see \cite{LionnetReisSzpruch2016}*{Theorem A.2}), and is of order $h$.
The following theorem thus asserts the convergence of the time-discretization scheme \eqref{equation---reference--FPS.scheme}.
	\begin{theorem}
	\label{proposition---MainTheorem}
		There exists a constant $C$ (not depending on $\pi$) such that 
			\begin{align*}
				\mathrm{ERR}^Y(\pi)^2 + \mathrm{ERR}^Z(\pi)^2 \le C \, h.
			\end{align*}
	\end{theorem}
We postpone of the proof of this theorem to Section \ref{sec:proofofConvTheorem}.
We present immediately the implications of the scheme we propose in terms of the numerical approximations it produces.

\section{Analysis and numerical simulations for the full scheme} 		
\label{section---numerical.scheme.and.results}

\subsection{Full scheme including approximation of the conditional expectations}
\label{sec:BasicApproxofExpectations}

In order to implement the scheme \eqref{equation---reference--FPS.scheme}, we need to also approximate the conditional expectations involved at each step. 
We use a quantization method (see \cites{Chassagneux2014,ChassagneuxRichou2016,SagnaPages2015}).
Let us describe briefly how it works in the case where the forward process $X$ is the standard Brownian motion, as will be the case in the examples. 

Over the time grid $\pi = (\ti)$, one replaces the Markov chain $(W_\ti)_{i=0, \ldots, N}$, by the tree-supported Markov chain $(\wb W_\ti)_{i = 0, \ldots, N}$, where $\wb W_\tip = \wb W_\ti + \wb{\Delta W}_\tip$. 
Here, the $\wb{\Delta W}_\tip$ are i.i.d. random variables with symmetric and finitely supported distribution matching the first moments of the $\cN(0,h)$ distribution of $\Delta W_\tip$.
More precisely, in our examples below, we choose a recombining trinomial tree, thus replacing $\cN(0,h)$ by the distribution putting weights $(p_{-1} \ p_{0} \ p_{1}) = (\frac16 \ \frac23 \ \frac16)$ on the points $g_{-1}=-\sqrt{3h}$, $g_0=0$ and $g_1=\sqrt{3h}$. It has the same moments as $\cN(0,h)$ of order $0$ up to order $q=5$. 
Let us also denote $x_l = l \, \sqrt{3h}$, for any $l \in \Z$. 

For the backward processes, we initialize the scheme with $\wb Y_N = \xi^N = g(\wb W_T)$ and then, for $i=N-1$ down to $0$, $\wb Y_\ip$ is used to compute $(\wb Y_i,\wb Z_i)$ as
	\begin{align*}	
		\wb Y_i = \bE_i\Big[ T^h(\wb Y_\ip) + f\big(T^h(\wb Y_\ip),\wb Z_i\big) \hip \Big]
		\quad \text{and} \quad
		\wb Z_i = \bE_i\Big[ T^h(\wb Y_\ip) \wbHip^* \Big].
	\end{align*}
Here, $\wbHip = T^{r^h}(\wb{\Delta W}_\tip)\, h^{-1}$ with $r^h$ defined immediately after \eqref{equation---Error.of.approximating.by.Hip}. 
For all $i$ there exists measurable functions $\wb y_i : \Gamma_i \rightarrow \R^n$ and $\wb z_i : \Gamma_i \rightarrow \R^{n \times d}$, where $\Gamma_i = \sqrt{3h} \cdot \{-i,\ldots,i\} = \{x_{-i},\ldots,x_i\}$, such that $\wb Y_i = \wb y_i(\wb W_\ti)$ and $\wb Z_i = \wb z_i(\wb W_\ti)$ due to the Markovian framework and properties of the conditional expectation.
Thus, at each $\ti \in \pi^N$ and $\wb W_\ti = x_k$, for $k \in \{-i, \ldots , i \}$, the conditional expectation reduces to a finite sum, 
	\begin{align*}
		\wb Y_i &= \wb y_i(\wb W_\ti = x_k) 
							= \sum_{j \in \{-1, 0, 1\}} p_j \, . \, \Big( T^h(\wb y_\ip(x_{k+j}) + f\big(T^h(\wb y_\ip(x_{k+j})),\wb z_i(x_k)\big) \ h	\Big)			
							\\
		\wb Z_i &= \wb z_i(\wb W_\ti = x_k) 
							= \sum_{j \in \{-1, 0, 1\}} p_j \, . \, \Big( T^h(\wb y_\ip(x_{k+j})) \ {T^{r^h}(g_j)}{h}^{-1} \Big) .
	\end{align*}
We state the result on the error for the approximation of the conditional expectations in Section \ref{sec:mainresultNumerics} below. We now present our numerical experiments to highlight the gains of our results.

\subsection{Numerical results} 

\subsubsection{Convergence testing}

We consider the FBSDE \eqref{equation---canonicSDE}-\eqref{equation---canonicBSDE}, with $T=1$, $x_0=0$, $b=0$ and $\sigma=1.5$ (so that $X=\sigma W$), $f(y,z)=-y^3$ and $g(x)={x}^2$. 
We approximate the solution using the standard implicit and explicit BTZ schemes, as well as the Full-Projection scheme, with $N \in \{5,10,15,20,30,40,50,60,70,80\}$ time steps. 
As we do not have an explicit solution to the BSDE \eqref{equation---canonicBSDE}, we use as a proxy the average of the values returned by the implicit scheme and the Full-Projection scheme when $N=120$.

Figure \ref{Figure---convergence} depicts the behaviour of the returned value $Y^N_0$ versus $N$, as well as the computation time versus the error.
\begin{figure}[th]
	\centering
	\subfigure
	{	\includegraphics[scale=0.39]{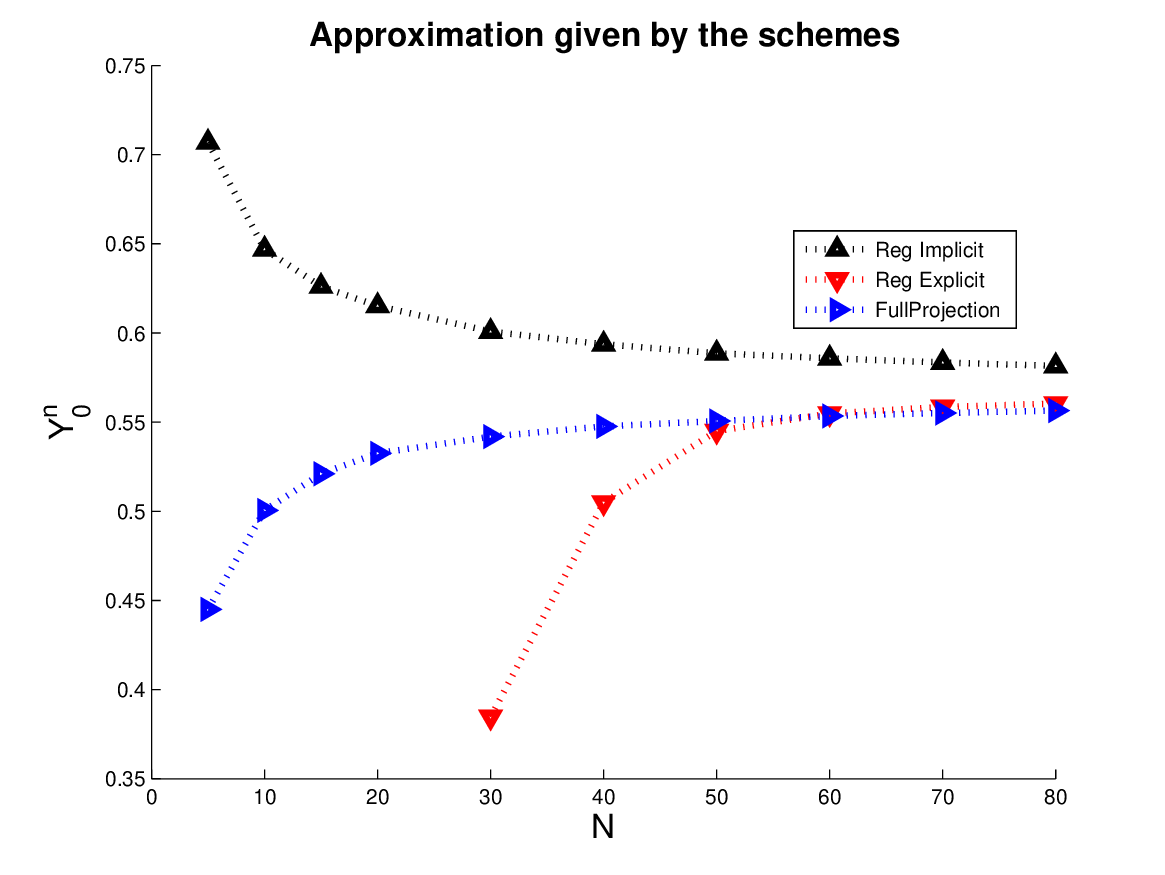}	}
	\subfigure
	{	\includegraphics[scale=0.39]{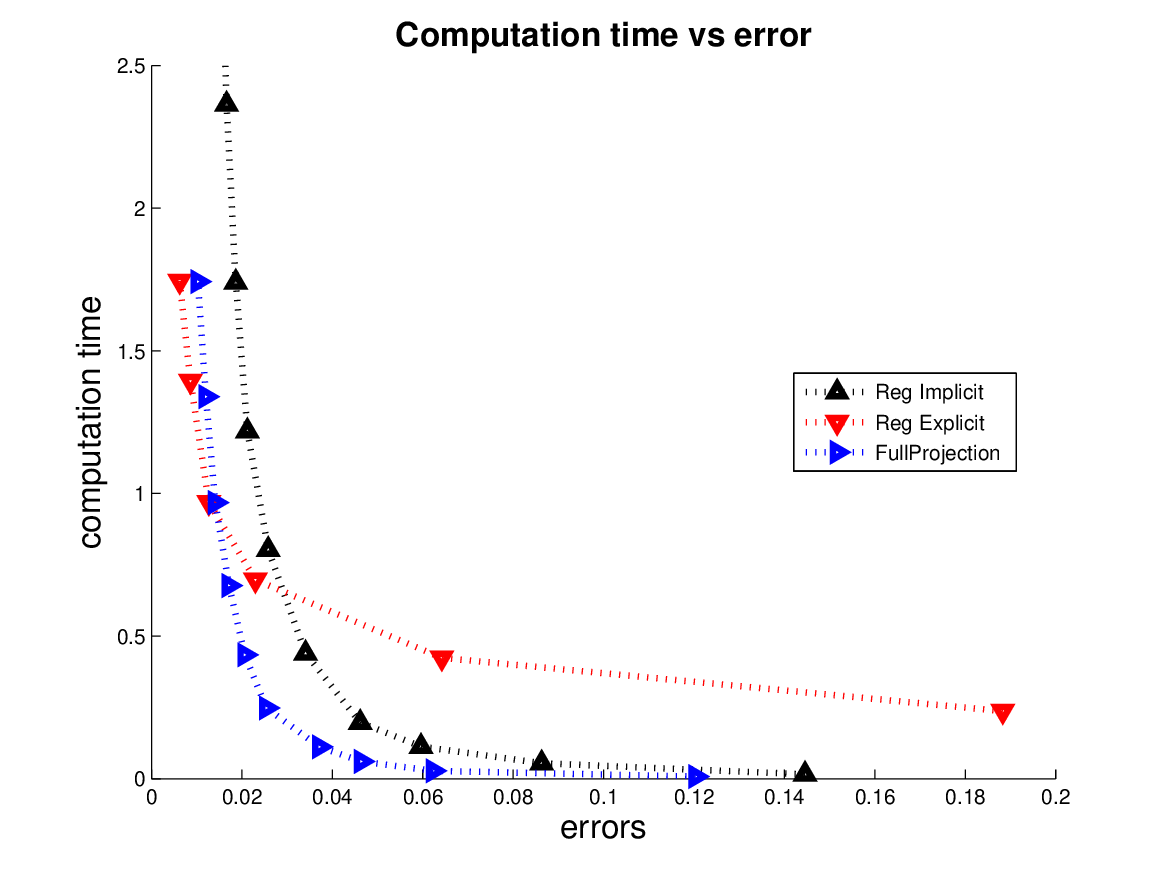}	}
	\caption{ Errors versus $N$ and computation time versus error. }
	\label{Figure---convergence}
\end{figure}
We see that the Full-Projection scheme, being an explicit scheme, provides similar errors as the implicit scheme but at a lower computational cost. However, it does not explode for values of $N$ that are too small (explosion characterized by a \emph{NaN} value and the missing points on the red curve).

\subsubsection{Numerical stability}

We now consider the FBSDE \eqref{equation---canonicSDE}-\eqref{equation---canonicBSDE}, with $T=1$, $x_0=0$, $b=0$ and $\sigma=2.5$ (so that $X=\sigma W$), $f(y,z)=-y-y^3$ and $g(x)=-7 \lor x \land 7$.
On Figure \ref{Figure---numerical.stability}, for various values of $N$, we plot the maximum and minimum of the random variable $Y^N_i$, for $i=0, \ldots, N$.
\begin{figure}[!ht]
	\centering
	\subfigure
	{	\includegraphics[scale=0.39]{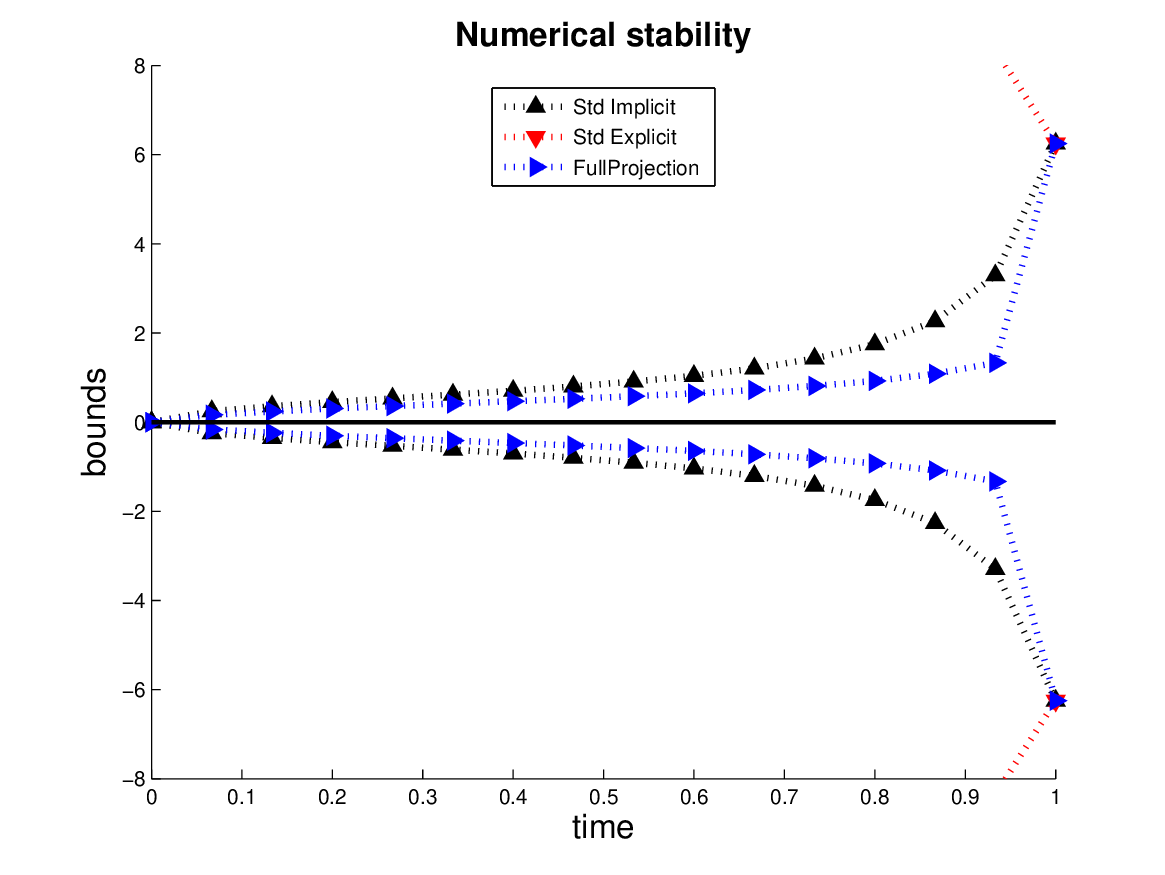}	}
	\subfigure
	{	\includegraphics[scale=0.39]{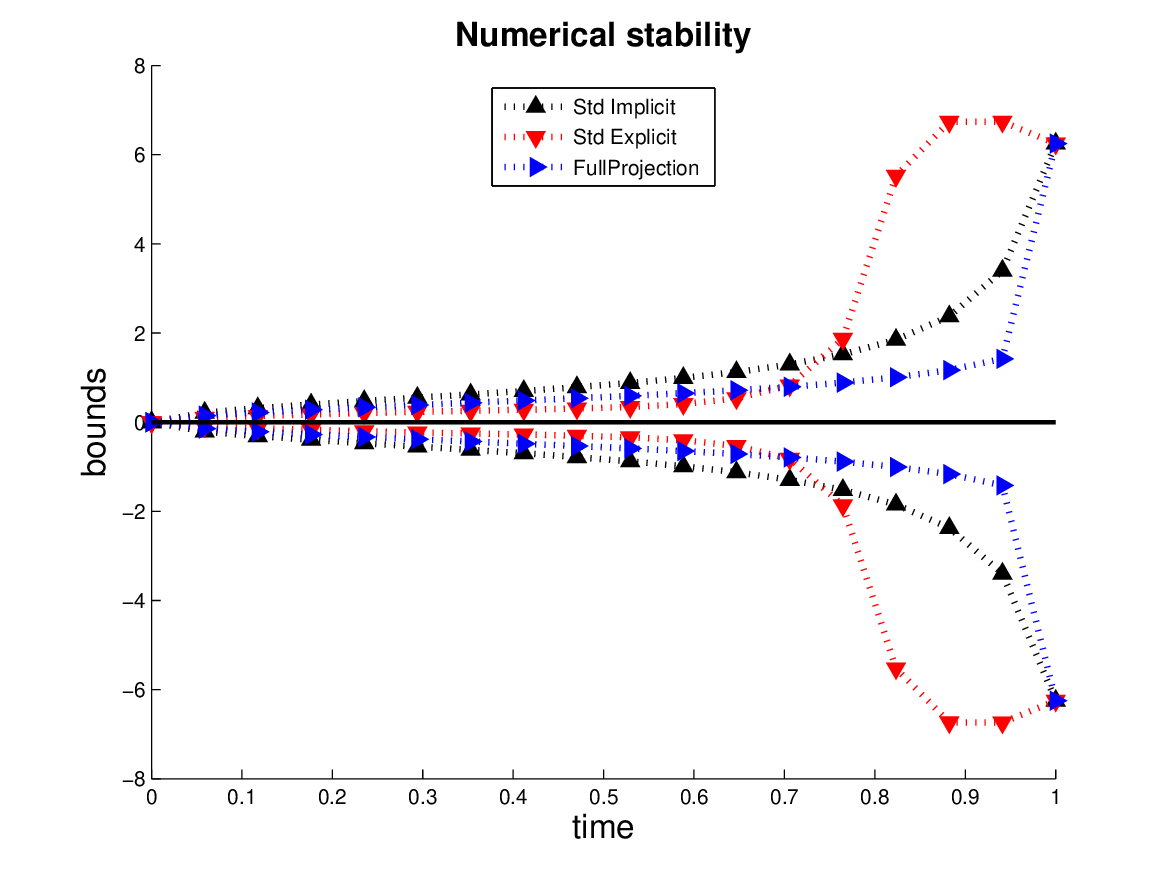}	}
	\subfigure
	{	\includegraphics[scale=0.39]{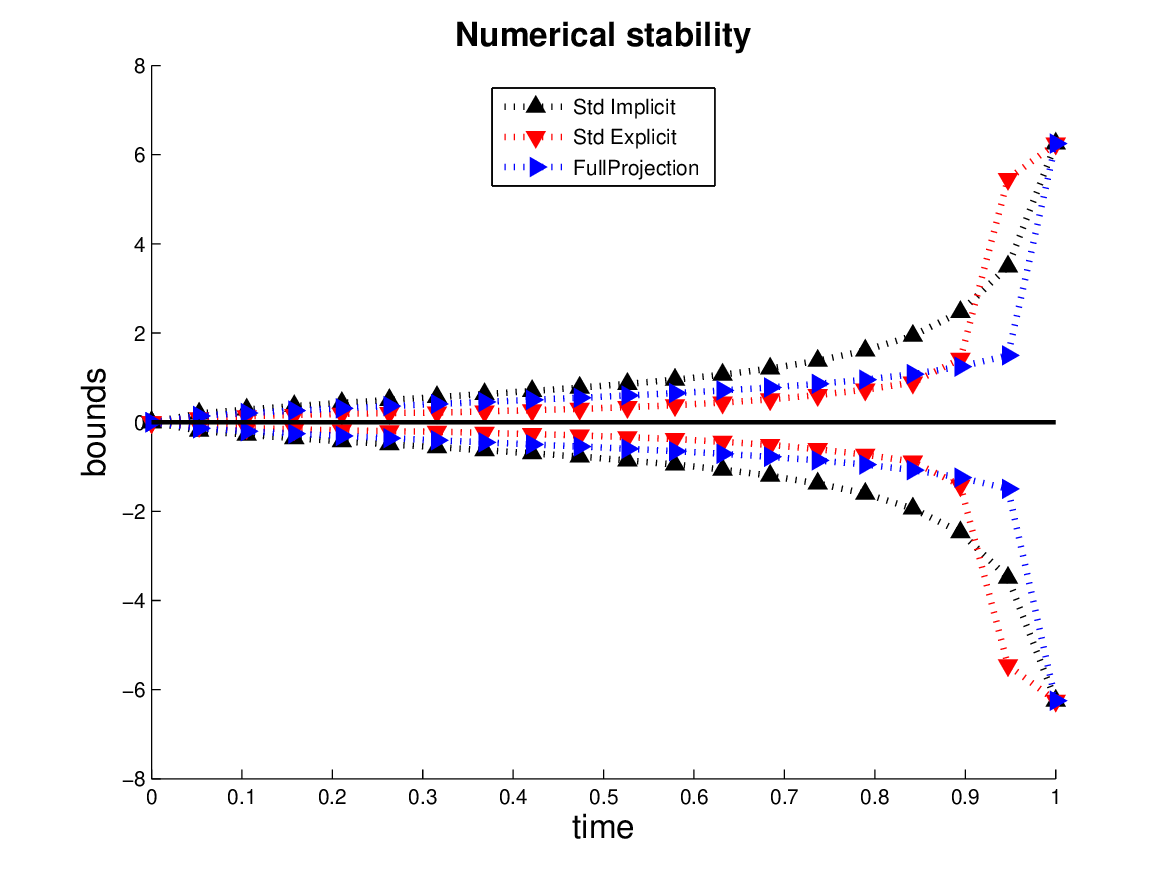}	}	
	\subfigure
	{	\includegraphics[scale=0.39]{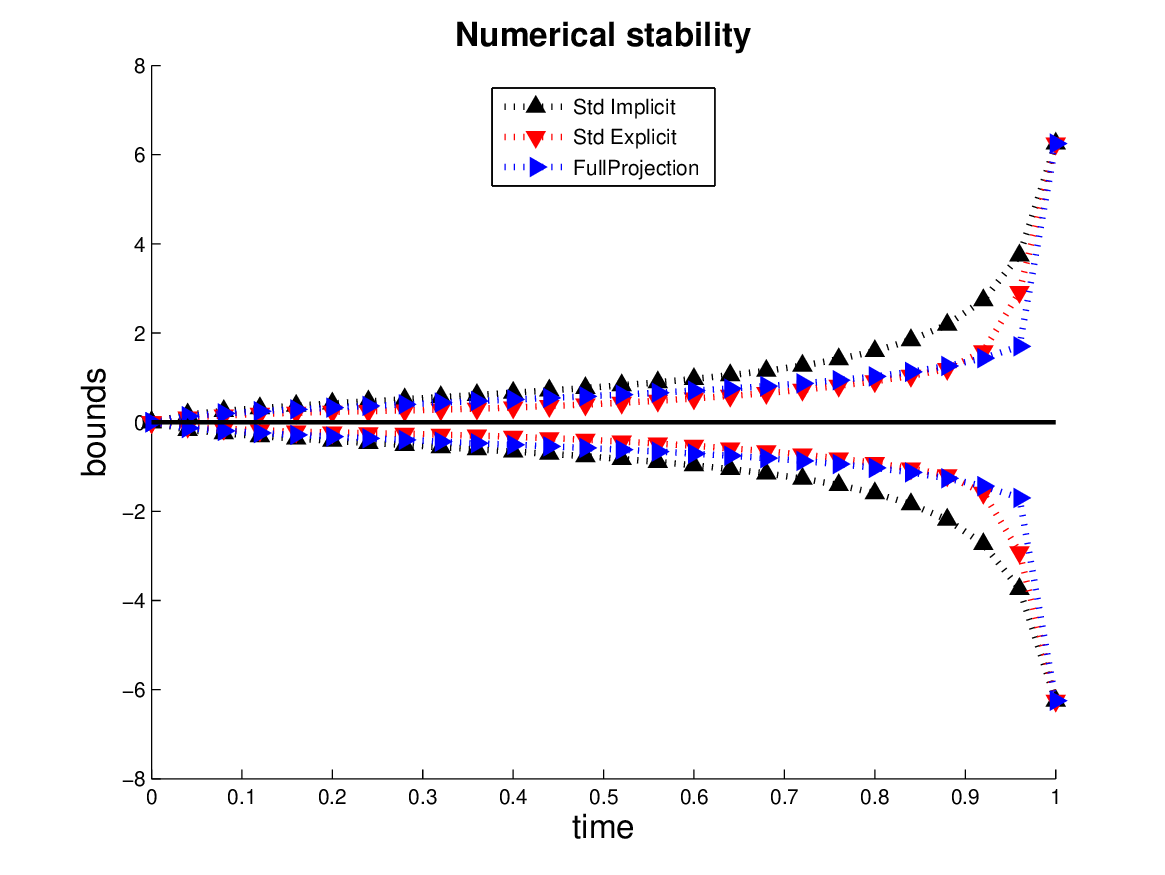}	}	
	\caption{Maximum and minimum processes associated with $(Y^N_i)_{i=0, \ldots, N}$, for $N=15$, $17$, $19$, $25$, read respectively from left-to-right, top-to-bottom.}
	\label{Figure---numerical.stability}
\end{figure}
We see that when the number of time steps is too small, the explicit may explode or not be numerically stable (characterized here by a non-convexity or non-concavity of the curves).

More precisely, we have here a driver satisfying \hfmon with $M_y = -1 < 0$.
Since $f(0)=0$, the process $(0,0)$ is solution for the terminal condition $g'=0$.
As a consequence, the continuous-time dynamics satisfies for some $c<0$ the bounds $\norm{Y_t}_\infty \le e^{c (T-t)} \norm{\xi}_\infty \le \norm{\xi}_\infty$.
Here, we see that for the implicit and the Full-Projection schemes, the corresponding bounds $\norm{Y^N_i}_\infty \le e^{c' (T-\ti)} \norm{\xi^N}_\infty \le \norm{\xi^N}_\infty$ are satisfied for some $c'<0$.
This is not the case for the explicit scheme.
When $N=15$ (top left picture) the bound $\norm{Y^N_i}_\infty \le \norm{\xi^N}_\infty$ is violated and the scheme in fact explodes. 
When $N=17$, this same bound is violated but eventually the scheme does not explode. This behaviour is specific to a BSDE (for an ODE, it could not happen), and is linked to interpreting the random variable through its $\infty$-norm, thus not taking into account its distribution.
When $N=19$, this bound is satisfied but the stronger bound $\norm{Y^N_i}_\infty \le e^{c' (T-\ti)} \norm{\xi^N}_\infty$ is not (or rather, it is, but with a much worse $c''$). 
It becomes the case only when $N$ is high enough.
By contrast, the implicit and the Full-Projection explicit schemes display this numerical stability property for all partitions, even the coarser ones.

\subsection{The general scheme and the error estimate}
\label{sec:mainresultNumerics}

Previous works on polynomial-growth FBSDEs \cites{LionnetReisSzpruch2015,LionnetReisSzpruch2016} focused only on estimating the time-discretization error. We carry here the analysis of the full error, due to both the time-discretization and the approximation of the conditional expectations. Unlike \cite{Chassagneux2014}, we do so for a general forward process $X$, which recovers the case of $X$ being a linear Brownian motion. 

In all generality, for a given $\eta > 0$, $x_0\in \bR^d$ and $M \in \N$, we consider the spatial grid
	\begin{align*}
		\Gamma = \big\{ x_0 + k \eta \ \Big|\ k \in \Z^d \ \text{ and } \ \abs{k}_\infty \le M\big\} .
	\end{align*}
We denote by $\Pi$ the projection from $\R^d$ on the grid $\Gamma$ in the distance associated with the $L^\infty$-norm (i.e. component-wise projection).
For all $x$ in the convex hull of $\Gamma$ we have  $\abs{x -\Pi(x)} \le \eta$.
The forward process is approximated as follows: it is initialized with $\wb{X}_0 = x_0$ and then, for $i=0, \ldots, N-1$ 
	\begin{align}		\label{equation---reference--X.scheme--fully.discretized}
		\wb X_\ip = \Pi \Big( \wb X_i + b(\ti,\wb X_i) h + \sigma(\ti,\wb X_i) \wb{\Delta W}_\tip \Big).
	\end{align}
For the backward processes, just as previously, we initialize the scheme with $\wb Y_N = \xi^N = g(\wb X_N)$ and then, for $i=N-1$ down to $0$, $\wb Y_\ip$ is used to compute $(\wb Y_i,\wb Z_i)$ as
	\begin{align}		
	\label{equation---reference-YZ.scheme--fully.discretized} 
		\wb Y_i = \bE_i\big[ T^h(\wb Y_\ip) + f\big(T^h(\wb Y_\ip),\wb Z_i\big) \, h \big]
		\quad\textrm{and}\quad
		\wb Z_i = \bE_i\big[ T^h(\wb Y_\ip) \wbHip^* \big].
	\end{align}

\paragraph*{}
We know from the nonlinear Feynman--Kac formula  (see \cite{LionnetReisSzpruch2015}) that the solution to the FBSDE \eqref{equation---canonicSDE}-\eqref{equation---canonicBSDE} can be represented as $Y_t = Y^{0,x_0}_t = y(t,X^{0,x_0}_t)$ and $Z_t = Z^{0,x_0}_t = z(t,X^{0,x_0}_t)$. The function $y$ here is the solution to the PDE
	\begin{align}
	\label{eq:PDEforErrorEstimate}
		\partial_t y + \half \trace\big[\partial^2_{xx}y \ \sigma \sigma^*\big] + \partial_x y \ b + f(y,\partial_x y \ \sigma) = 0 \ \ \text{on }\ [0,T] \times \R^d
	\end{align}
with terminal condition $y(T,\cdot)=g$, and $z=\partial_x y \, \sigma$. 

The error estimate for the full scheme is then the following. 
	\begin{theorem} 	
	\label{proposition---weak.error.estimate.global}
	Assume that $f$, $y$ and $z$ are of class $C^3$ with polynomially growing derivatives. Then there exists a $C \ge 0$ such that 
			\begin{align*}
				\abs{Y_0 - \wb Y_0} \le C h + C \eta h^{-1}.
			\end{align*}
	\end{theorem}
The proof of the theorem is carried out in Section \ref{sec:proofofmainnumericstheorem}.

\section{Analysis of the time-discretization scheme}		
\label{section---analysis.of.the.time.discretization.scheme} 


To prove the time-discretization's convergence rate we first need to prove a number of integrability and stability properties for the numerical scheme. 
Another and crucial ingredient is a stability property for the associated operators $\Phi_i$ (cf Section \ref{sec:fullprojectionscheme}). Roughly speaking, we want an estimate of the form $\norm{\Phi_i(\cY)-\Phi_i(\cY')} \le e^{c h} \norm{\cY-\cY'}$, which controls the error of the outputs solely by the contribution of the input errors. This then allows to bound the global error by a sum of local errors. Moreover, due to the features of our Full-projection scheme, by truncating the inputs $\cY$, we are able to show our scheme satisfies this sought stability estimate in a ``pure form'', i.e. without additional imperfection terms which then have to be controlled as in \cite{LionnetReisSzpruch2016} (cf Remark \ref{remark---true.stability}). 
The third ingredient, the local errors, are studied in Section \ref{sec:LocalTime-DiscretizationError}.

\subsection{Non-explosion}
\label{sec:non-explosion}

This first result provides an estimate for growth of the scheme outputs as a function of the inputs after one-step of the scheme.

	\begin{proposition}		\label{proposition---size.estimate.one.step}
		Let $h \le (16 (d+1) L_z^2)^{-1}$ and $\alpha \le \frac{1}{2(m-1)}$. Then, there exists a constant $c \in \R$ and $K \ge 0$ (not depending on $h$) such that, for all $i \in \{0,\ldots,N-1\}$, 
			\begin{align}		
			\label{equation---size.estimate.one.step.almost.sure}
				\abs{Y_i}^2 + \frac{1}{8} \abs{Z_i}^2 h
					\le e^{c h} \bE_i\big[ \abs{T^h(Y_\ip)}^2 \big] + K^2 h .
			\end{align}
		This estimate implies for both schemes \eqref{equation---reference--FPS.scheme} and \eqref{equation---reference--FPS.scheme.alternative}
			\begin{align*}
				&\abs{Y_i}^2 + \frac{1}{8} \abs{Z_i}^2 h
					\le e^{c h} \bE_i\big[ \abs{Y_\ip}^2 \big]	+ K^2 h			\hspace{0.5cm} \text{and} \hspace{0.5cm} 
				\abs{\wt Y_i}^2 + \frac{1}{8} \abs{\wt Z_i}^2 h
					\le e^{c h} \bE_i\big[ \abs{\wt Y_\ip}^2 \big] + K^2 h .
			\end{align*}
	\end{proposition}

\begin{proof}
		Similarly to \cite{LionnetReisSzpruch2016}, we rewrite equation \eqref{equation---reference--FPS.scheme} as
			\begin{align*}
				Y_i + \Delta M_\ip = T^h(Y_\ip) + f(T^h(Y_\ip),Z_i) h ,
			\end{align*}
		where $\Delta M_\ip = T^h(Y_\ip) + f(T^h(Y_\ip),Z_i) h - \bE_i[T^h(Y_\ip) + f(T^h(Y_\ip),Z_i) h]$.
		We square both sides and exploit martingale property of $\Delta M_\ip$ to get
			\begin{align*}
				\abs{Y_i} + \bE_i[\abs{\Delta M_\ip}^2] 
				= \bE_i\Big[ \abs{T^h(Y_\ip)}^2 
				+ 2 \scalar{T^h(Y_\ip)}{f(T^h(Y_\ip),Z_i) h} 
				+ \abs{f(T^h(Y_\ip),Z_i)}^2 h^2 \Big] .
			\end{align*}
		Thus \eqref{equation---assumptions.growth} implies, for $\alpha_z > 0$ to be fixed later,		
			\begin{align*}
				\abs{Y_i} & + \bE_i[ \abs{\Delta M_\ip}^2] \\
						&\le  \bE_i\bigg[ \Big\{ 1 + 2\big( M_y + \alpha_z\big) h \Big\} \abs{Y_\ip}^2 \bigg]
										+ \frac{L_z^2}{\alpha_z} \abs{Z_i}^2 h + \bE_i \big[ \abs{f(T^h(Y_\ip),Z_i)}^2 \big] h^2 +\frac{\abs{f(0,0)}^2}{\alpha_z} h  .
			\end{align*}
		In exactly the same way as in \cites{LionnetReisSzpruch2016,Lionnet2016} one can show that		
			\begin{align*}
				\bE_i\big[\abs{\Delta M_\ip}^2\big] 
					\ge \half \abs{Z_i}^2 \Lambda^{-1} h -  d h^2 \bE_i\Big[ \abs{ f(T^h(Y_\ip),Z_i) }^2  \Big].
			\end{align*}
		Next, using \hfregY and \hfreg, 
			\begin{align*}
				\bE_i \big[ \abs{f(T^h(Y_\ip),Z_i)}^2 \big] h^2 
				\le 4 L_y^2 \bE_i \big[ ( 1 + \abs{T^h(Y_\ip)}^{2(m-1)} ) \abs{T^h(Y_\ip)}^{2} \big] h^2 + 2 L_z^2 \abs{Z_i}^2 h^2  + \abs{f(0,0)}^2 h^2.
			\end{align*}
		Therefore, using  $\Lambda^{-1} \ge 1$ and properties of the projection, we have the estimate
			\begin{align*}
				\abs{Y_i}^2 + \bigg( \half - \frac{L_z^2}{\alpha_z} \bigg) \abs{Z_i}^2 h  
					&\le  \bE_i\bigg[ \Big\{ 1 + 2 \big( M_y + \alpha_z\big) h \Big\} \abs{Y_\ip}^2 \bigg] + 2(d+1) L_z^2 \abs{Z_i}^2 h^2
					\\
						&\hspace{3cm} + 4 (d+1) L_y^2 \bE_i \big[ ( 1 + (R^h)^{2(m-1)} ) \abs{Y_\ip}^{2} \big] h^2  + K^2 h,
			\end{align*}
		where $K= (d+1)\abs{f(0,0)}^2 h + \frac{\abs{f(0,0)}^2}{\alpha_z}$.
		Now, we choose $\alpha_z = 4 L_z^2$, 
		so that $\half - \frac{L_z^2}{\alpha_z} = \frac{1}{4}$.
		We take $h$ such that $2(d+1) L_z^2 h \le \frac{1}{8}$.
		With this, we have	
			\begin{align*}
				\abs{Y_i}^2 + \frac{1}{8} \abs{Z_i}^2 h
					\le e^{c(R^h) h} \bE_i\big[ \abs{T^h(Y_\ip)}^2 \big] + K^2 h,
			\end{align*}
		where, with $R^h = R_0 h^{-\alpha}$, we have 
$				c(R^h) 
				= 2M_y + 8 L_z^2 + 4 (d+1) L_y^2 (1 + R_0^{2(m-1)} h^{-2(m-1)\alpha}) h.$
		Choosing $\alpha \le \frac{1}{2(m-1)}$ ensures $c(R^h)$ remains bounded as $h \goesto 0$, in fact, $c(R^h) \goesto 2M_y + 8 L_z^2$ if $\alpha < \frac{1}{2(m-1)}$.
		So there exists a constant such that $c(R^h) \le c$, which proves the main statement. 
		
		Now, using $\abs{T^h(Y_\ip)} \le \abs{Y_\ip}$ we obtain the second estimate. 
		For the third, we recall that $\wt Y_\ip = T^h(Y_\ip)$, $\wt Z_i = Z_i$, $\wt Y_i = T^h(Y_i)$, and use $\abs{T^h(\wt Y_i)} \le \abs{\wt Y_i}$.
	\end{proof}
Iterating this one-step estimate one easily obtains a path-wise and an $L^p$ estimate showing that the scheme does not blow-up.
	\begin{corollary}
		There exists a constant $c \in \R$ such that, for all $i \in \{0,\ldots,N-1\}$, 
			\begin{align}		\label{equation---size.estimate.all.steps.almost.sure}
				\abs{Y_i}^2 + \frac{1}{8} \bE_i\bigg[ \sum_{j=i}^{N-1} \abs{Z_j}^2 h \bigg]
					\le e^{c (T-\ti)} \bE_i\big[ \abs{T^h(Y_N)}^2 \big] + e^{c(T-\ti)} K^2 (T-\ti) .
			\end{align}
			Moreover, for any $p \geq 1$ there exists a constant $C \ge 0$ (not depending on $h$) such that
			\begin{align}	\label{equation---size.estimate.all.steps.Lp}
				\max_{i=0, \ldots, N} \bE\big[\,\abs{Y_i}^{2p}\big] + \bE\bigg[ \sum_{i=0}^{N-1} ( \abs{Z_i}^2 h )^p \bigg] \le C .
			\end{align}	
	\end{corollary}

	\begin{proof}
		Estimate \eqref{equation---size.estimate.all.steps.almost.sure} follows by iterating \eqref{equation---size.estimate.one.step.almost.sure} 
		(with Lemma A.3 in \cite{LionnetReisSzpruch2016}) using that $\abs{T^h(y)} \le \abs{y}$. 
		The result in \eqref{equation---size.estimate.all.steps.Lp} is proved in the same fashion as Proposition 3.5 in \cite{LionnetReisSzpruch2016}.
	\end{proof}

\subsection{Stability and numerical stability of the scheme} 
\label{sec:stability}
In this subsection, we consider the random variables $Y^1_\ip$ and $Y^2_\ip$ as inputs at time $\tip$, and the respective outputs $({Y}^1_i,{Z}^1_i)$ and $(Y^2_i,Z^2_i)$ at time $\ti$ from one step of the scheme \eqref{equation---reference--FPS.scheme}. The next result states that the error of the outputs is controlled solely by the contribution of the input errors (as opposed to the setting in \cite{LionnetReisSzpruch2016} --- cf Remark \ref{remark---true.stability}). 

We denote generically by $\delta x$ the quantity ${x}^1-x^2$.

	\begin{proposition}		
	\label{proposition---stability.of.the.scheme} 
		There exists a constant $c\geq 0$ (independent on $\pi$, $Y^1_\ip$ and $Y^2_\ip$) such that,
			\begin{align}		
			\label{equation---stability.estimate.one.step.almost.sure}
			\textrm{for any }
			h \le h_0 = \frac{1}{8} \ \frac{1}{2(d+1) L_z^2}
			\quad \textrm{we have} \quad
				\abs{\delta Y_\ip}^2 + \frac{1}{8} \abs{\delta Z_i}^2 \hip 
						\le e^{c h } \bE_i\big[ \abs{\delta Y_\ip}^2 \big] .
			\end{align}
	\end{proposition}

	\begin{proof}
		This proof is fairly similar to that of Proposition \ref{proposition---size.estimate.one.step}. 
		So we skip most of the estimations and readily conclude that
			\begin{align*}
				\abs{\delta Y_\ip}^2 + \frac{1}{8} \abs{\delta Z_i}^2 \hip 
						\le e^{c(R^h) h } \bE_i\big[ \abs{T^h(Y^1_\ip)-T^h(Y^2_\ip)}^2 \big] ,
			\end{align*}
		where, with $R^h = R_0 h^{-\alpha}$, we have
	$		
				c(R^h) 
					= 2 M_y + 4 L_z^2 + 3 (d+1) L_y^2 (1 + 2 R_0^{2(m-1)} h^{-2(m-1)\alpha}) h$. 
		So $c(R^h)$ is a bounded function of $h$ as soon as $\alpha \le \frac{1}{2(m-1)}$. 
		(And it converges to $2 M_y + 4 L_z^2$ as $h \goesto 0$ if $\alpha < \frac{1}{2(m-1)}$.)
		Finally, since $T^h$ is $1$-Lipschitz, we have the sought result.
	\end{proof}

	\begin{remark}[True stability] \label{remark---true.stability}
		Proposition \ref{proposition---stability.of.the.scheme} states that our explicit scheme is actually stable in the standard sense.
		In \cite{LionnetReisSzpruch2015}, the predominantly implicit $\theta$-schemes (with $\theta \in [1/2,1]$) were \emph{asymptotically stable}. 
		Indeed, they had an additional term $R_i$ on the RHS of the estimate \eqref{equation---stability.estimate.one.step.almost.sure} 
		and the \emph{stability remainder} defined as $\cR^S = \max_{i} \sum_{j=i}^{N-1} e^{c(j-i)h} \bE[R_j]$ would vanish as $\abs{\pi}=h \goesto 0$.
		The modified explicit schemes of \cite{LionnetReisSzpruch2016} with tamed drivers $f^h$ were \emph{almost stable}. 
		Indeed, they had a \emph{stability imperfection} term $\cI_i$ on the RHS of the estimate \eqref{equation---stability.estimate.one.step.almost.sure}, 
		satisfying $\bE[\cI_i] \le C h^{\mu+1}$, which in particular made those schemes asymptotically stable.
		By contrast, the {Full Projection Scheme} under study satisfies the standard $L^2$-stability estimate, with no additional term.
	\end{remark}

	\begin{remark}
		The only properties of $\Hip$ used to prove Proposition \ref{proposition---size.estimate.one.step} and \ref{proposition---stability.of.the.scheme} are (a) and (b).
	\end{remark}

Generally speaking, a numerical scheme is said to be \emph{numerically stable} if it preserves some qualitative properties of the continuous-time dynamics 
(note that numerical stability is different from the ($L^2$-)stability of the scheme established in Proposition \ref{proposition---stability.of.the.scheme}). 
Here we consider the following notion of numerical stability.
When the driver $f$ satisfies \hfmon with $M_y < 0$ and $L_z^2 < -2M_y$, we have, for any two terminal conditions $\xi^1$ and $\xi^2$, denoting by $(Y^1,Z^1)$ and $(Y^2,Z^2)$ the associated solutions to the continuous-time BSDE \eqref{equation---canonicBSDE}, the distance estimate
	\begin{align*}
		\norm{Y^1_t - Y^2_t}_2 \le e^{c (T-t)} \norm{\xi^1-\xi^2}_2,
	\end{align*}
with $c = M_y + L_z^2/2$ and $\norm{\cdot}_2$ the norm in $L^2(\F_t)$. In particular, since $c<0$ the continuous dynamics is \emph{contracting}.
Here, we can deduce from Proposition \ref{proposition---stability.of.the.scheme} that when $4L_z^2 \le -M_y$, $\alpha < \frac{1}{2(m-1)}$ and 
	\begin{align*}
		h \le \min\left\{ \frac{-M_y}{4} \frac{1}{3(d+1)L_y^2} , \Big( \frac{-M_y}{4} \frac{1}{6(d+1)L_y^2 R_0^{2(m-1)}} \Big)^{\frac{1}{1-2(m-1)\alpha}} \right\},
	\end{align*}
then
	\begin{align*}
		c(R^h) &= 2 M_y + 4 L_z^2 + 3 (d+1) L_y^2 h + 6 (d+1) L_y^2 R_0^{2(m-1)} h^{1-2(m-1)\alpha} 		\\
			&\le 2 M_y + (-M_y) + \frac{-M_y}{4} + \frac{-M_y}{4}	= \frac{M_y}{2} = c' < 0;
	\end{align*}
so the discretized dynamics, with the Full-Projection Scheme, is also contracting.
If the stability estimate \eqref{equation---stability.estimate.one.step.almost.sure} contained additional terms (cf Remark \ref{remark---true.stability}), this numerical stability criterion could not be obtained. Our stability results strongly improve upon those of \cite{LionnetReisSzpruch2016}.

To further our point, we argue via a size estimate for the solution to the continuous-time BSDE. 
Let $f$ satisfy $f(0,0)=0$, \hfmon with $M_y < 0$ and $L_z^2 < -2M_y$. Then for any square-integrable terminal conditions $\xi=g(X_T) \neq 0$, denoting by $(Y,Z)$ the associated solution to \eqref{equation---canonicBSDE}, we have
	\begin{align*}
		\norm{Y_t}_2 \le e^{c (T-t)} \norm{\xi}_2 < \norm{\xi}_2.
	\end{align*}
Here, when $8L_z^2 \le -M_y$, $\alpha < \frac{1}{2(m-1)}$ and 
	\begin{align*}
		h \le \min\left\{ \frac{-M_y}{4} \frac{1}{4(d+1)L_y^2} , \Big( \frac{-M_y}{4} \frac{1}{4(d+1)L_y^2 R_0^{2(m-1)}} \Big)^{\frac{1}{1-2(m-1)\alpha}} \right\},
	\end{align*}
we can deduce from Proposition \ref{proposition---size.estimate.one.step} (see also \eqref{equation---assumptions.growth}), by backward iteration, that for all $i \le N-1$, 
	\begin{align*}
		\norm{Y_i}_2 \le e^{c' (T-\ti)} \norm{\xi^N}_2 < \norm{\xi^N}_2,\qquad \textrm{with }\ c' = \frac{M_y}{2} < 0.
	\end{align*} 
The explicit scheme with modified driver in \cite{LionnetReisSzpruch2016} does not satisfy this strong numerical stability. Indeed, in the one-step size estimate of its Proposition 3.2, the dominant term in the constant in the exponential is some $\overline M_y$, which is only allowed to be $\ge 0$ (see Appendix B.2.2 and B.3.1, $\overline M_y = \max(M_y,0)$). That scheme thus only satisfies the property that, for all $i$, $\norm{Y_i}_2 \le \norm{\xi^N}_2$.

\subsection{Local time-discretization errors}  
\label{sec:LocalTime-DiscretizationError}

We look here at the error created by one step of the scheme compared to the BSDE dynamics.
Therefore, given $Y_\tip$ as input, we look at the random variables $(\wh Y_i,\wh Z_i) = \Phi_i(Y_\tip)$, namely 
	\begin{align}		
	\label{equation---definition.of.hatYi.hatZi}
		\wh Y_i = \bE_i\Big[ T^h(Y_\tip) + f\big(T^h(Y_\tip),\wh Z_i\big) h \Big]		
		\quad\textrm{and}\quad
		\wh Z_i = \bE_i\big[  T^h(Y_\tip) \Hip^* \big] .	
	\end{align}
We are interested in how these differ from $(Y_\ti,\wb Z_\ti) = \Psi_i(Y_\tip)$. 
Here, given a random variable $\cY_\tip \in L^p(\F_\tip)$, we denote by $(\cY_t,\cZ_t)_{t\in [\ti,\tip]}$ the solution over $[\ti,\tip]$ to the BSDE with driver $f$ and terminal condition $\cY_\tip$ at time $\tip$, and define $\Psi_i(\cY_\tip) = (\cY_\ti,\wb{\cZ}_\ti) = (\cY_\ti,h^{-1}\bE_i[\int_\ti^\tip \cZ_t dt])$.
The local time-discretization errors are the $\F_\ti$-measurable random variables
	\begin{align*}
		\tau^Y_i = Y_\ti - \wh Y_i 
		\qquad \text{and} \qquad 
		\tau^Z_i = \wb Z_\ti - \wh Z_i,
	\end{align*}
that is to say, $(\tau^Y_i,\tau^Z_i) = \Psi_i(Y_\tip) - \Phi_i(Y_\tip)$.
The lemma below gives an almost sure estimate on these local errors.
In it, we use the notation for $i\in\{0,\cdots,N-1\}$,
	\begin{align*}
		\mathrm{REG}^{Y,4}_i(h) = \bE_i\Big[ \sup_{\ti \le u \le \tip} \abs{Y_u - Y_\tip}^4 \Big]^\half
		\quad \text{and} \quad 
		\mathrm{REG}^{Z,2}_i(h) = \bE_i\bigg[  \int_\ti^\tip \abs{Z_u - \overline{Z}_\ti}^2 \ud u \bigg] .
	\end{align*}
Also, $C$ will denote a constant, whose value may change from instance to instance, but always depending only on the parameters of the problem ($g$, $f$ and the constants involved in the assumptions on it, etc) but not on $\pi$. 
Similarly, $\cP_i(\cY)$, for a random variable $\cY$ will denote the quantity $C(1+\bE_i[\abs{\cY}^p])$ for some constant $C$ and $p$, not depending on $\pi$, which may vary from line to line.

	\begin{lemma}			\label{proposition---error.time.discretization.local.for.RV}
		For all $i \in \{0,\ldots,N-1\}$, 
			\begin{align*}
				\abs{\tau^Z_i}^2 
						\le C \bE_i\bigg[ \int_\ti^\tip \abs{Z_u}^2 \ud u \bigg] 
							+ C \bE_i\bigg[ \int_\ti^\tip \abs{f(Y_u,Z_u)}^2 \ud u \bigg] 
							+ \cP_i(Y_\tip) h 
			\end{align*}
		and 
				$\abs{\tau^Y_i}^2 \
						\le \cP_i\big( \sup_{\ti \le u \le \tip} \abs{Y_u})_{}\big) h^2 \, \mathrm{REG}^{Y,4}_i(h)^\half 
							+ C h \mathrm{REG}^{Z,2}_i(h) + C h^2 \abs{\tau^Z_i}^2
							+ \cP_i(Y_\tip) h^3$. 
	\end{lemma}

	\begin{proof}
		\emph{Estimate for $\tau^Z_i$}. 
		We first write 
		\begin{align*}
				\tau^Z_i 
					&= \bigg( \frac1h \bE_i\bigg[ \int_\ti^\tip Z_u du \bigg] - \bE_i\Big[ Y_\tip H_\ip^* \Big] \bigg)
										+ \bigg( \bE_i\Big[ Y_\tip H_\ip^* \Big] - \bE_i\Big[ T^h(Y_\tip) H_\ip^* \Big] \bigg)
							= \tau^{Z,d}_i + \tau^{Z,t}_i.
			\end{align*}
		Similarly to \cites{LionnetReisSzpruch2015,LionnetReisSzpruch2016}, the first term can be estimated as 
			\begin{align*}
				\Abs{ \tau^{Z,d}_i }^2 
					&\le C \bE_i\bigg[ \int_\ti^\tip \abs{Z_u}^2 \ud u \bigg] \bE\bigg[ \Abs{\frac{\Delta W_\tip}{h} - H_\ip }^2 \bigg] + C \bE_i\bigg[ \int_\ti^\tip \abs{f(Y_u,Z_u)}^2 \ud u \bigg]	\\
					&\le C \bE_i\bigg[ \int_\ti^\tip \abs{Z_u}^2 \ud u \bigg] + C \bE_i\bigg[ \int_\ti^\tip \abs{f(Y_u,Z_u)}^2 \ud u \bigg],
			\end{align*}
		using assumption (c) on $\Hip$.
		For the second term, we use H\"older's inequality, Markov's inequality with a $p \ge 1$, and the property that $T^h(y)=y$ for $\abs{y} \le R^h$ and $\abs{T^h(y)-y} \le \abs{y}$. We obtain 
			\begin{align*}
				\Abs{ \tau^{Z,t}_i }^2 
					&\le \bE_i\Big[ \Abs{Y_\tip - T^h(Y_\tip)}^2\Big] \bE_i\Big[ \abs{H_\ip}^2 \Big]
					\\
					&\le \bE_i\Big[ \Abs{Y_\tip}^2 1_{\{\abs{Y_\ip} > R^h\}}\Big] \frac{\Lambda d}{h}
					\le \bE_i\Big[ \Abs{Y_\tip}^4 \Big]^\half  \bE_i\Big[ \abs{Y_\ip}^p \Big]^\half (R^h)^{-p/2} \frac{\Lambda d}{h}
					\\
					&\le C \bE_i\Big[ \Abs{Y_\tip}^4 \Big]^\half  \bE_i\Big[ \abs{Y_\ip}^p \Big]^\half h^{\alpha p/2 - 1},
			\end{align*}
		since $\Lambda \le 1$, and $R^h = R_0 h^{-\alpha}$.
		We take $p = \frac{4}{\alpha} \ge \frac{4}{\frac{1}{2(m-1)}}=8(m-1)$ so that $\alpha p/2 - 1 = 1$.
		Since $p \ge 4$, we know from the H\"older inequality that $\bE_i[\abs{Y_\tip}^4]^{1/4} \le \bE_i[\abs{Y_\tip}^p]^{1/p}$ and then, since $p \ge 8$ (we assumed $m \ge 2$),
			\begin{align*}
				\bE_i\Big[ \Abs{Y_\tip}^4 \Big]^\half  \bE_i\Big[ \abs{Y_\ip}^p \Big]^\half 
					&\le \bE_i\Big[ \abs{Y_\ip}^p \Big]^{2/p} \bE_i\Big[ \abs{Y_\ip}^p \Big]^\half
					\le \bE_i\Big[ \abs{Y_\ip}^p \Big],
			\end{align*}
		which proves the stated result for $\tau^Z_i$. 

	\paragraph*{}
	\emph{Estimate for $\tau^Y_i$.} 
		We first write
			\begin{align*}
				\tau^Y_i 
					&= \bE_i\bigg[ \int_\ti^\tip f(Y_u,Z_u) \ud u - f(Y_\tip,\widehat{Z}_i) h \bigg]																							\\
							&\hspace{2cm}+ \bE_i\big[ Y_\tip + f(Y_\tip, \wh Z_i) h \big] - \bE_i\big[ T^h(Y_\tip) + f\big(T^h(Y_\tip),\wh Z_i) h \big] 				
					= \tau^{Y,d}_i + \tau^{Y,t}_i.
			\end{align*}
		Similarly to \cites{LionnetReisSzpruch2015,LionnetReisSzpruch2016}, the first term can be estimated as 
			\begin{align*}
				\Abs{ \tau^{Y,d}_i }^2 
					\le C h^2 \big( 1 + \bE_i\big[\sup_{\ti \le u \le \tip} \abs{Y_u}^{4(m-1)} \big] \big) \mathrm{REG}^{Y,4}_i(h)^\half + C h \mathrm{REG}^{Z,2}_i(h) + C h^2 \abs{\tau^Z_i}^2.
			\end{align*}
		For the second term, we use \hfregY, the Hï¿½lder inequality and $\abs{T^h(y)} \le \abs{y}$ to estimate
			\begin{align*}
				\Abs{ \tau^{Y,t}_i }^2 
					&\le C \bE_i\big[ \abs{ f(Y_\tip, \wh Z_i) - f\big(T^h(Y_\tip),\wh Z_i) }^2 \big] h^2 
					+ C \bE_i\big[ \abs{ Y_\tip - T^h(Y_\tip) }^2 \big]
					\\
					&\le C h^2 \bE_i\big[ L_y^2 (1+\abs{Y_\tip}^{m-1}+\abs{T^h(Y_\tip)}^{m-1})^2 \ \abs{Y_\tip - T^h(Y_\tip)}^2 \big] + C \bE_i\big[ \abs{Y_\ip-T^h(Y_\ip)}^2 \big] 						\\
					&\le C h^2 \bE_i\big[ (1+2\abs{Y_\tip}^{m-1})^4 \big]^\half  \bE_i\big[ \abs{Y_\tip - T^h(Y_\tip)}^4 \big]^\half + C \bE_i\big[ \abs{Y_\ip-T^h(Y_\ip)}^2 \big].
				\end{align*}
		Using $(a+b)^q \le 2^{q-1} (a^q + b^q)$ and the Markov inequality with with powers $2p$ and $p$, as previously, 
			\begin{align*}
				\Abs{ \tau^{Y,t}_i }^2 
					&\le C h^2 \bE_i\big[ (1+\abs{Y_\tip}^{4(m-1)}) \big]^\half  \bE_i\big[ \abs{Y_\tip}^8 \big]^{1/4} \bE_i\big[\abs{Y_\tip}^{2p}\big]^{1/4} h^{\alpha \, 2p /4}				\\
					&\hspace{5cm}
								+ C \bE_i\big[ \abs{Y_\tip}^4 \big]^{1/2} \bE_i\big[\abs{Y_\tip}^p\big]^{1/2} h^{\alpha p /2}
								\\
					&\le C h^2 \big(1 + \bE_i\big[\abs{Y_\tip}^{4(m-1)}\big] \big)^\half  \bE_i\big[ \abs{Y_\tip}^8 \big]^{1/4} \bE_i\big[\abs{Y_\tip}^{2p}\big]^{1/4} h^{\alpha p /2}				\\
					&\hspace{5cm}					
								+ C \bE_i\big[ \abs{Y_\tip}^4 \big]^{1/2} \bE_i\big[\abs{Y_\tip}^p\big]^{1/2} h^{\alpha p /2}.
			\end{align*}
		We take $p = \frac{6}{\alpha} \ge 12(m-1)$ so that $\alpha p/2 = 3$. 
		Since $2p \ge 8$ and $p \ge 4$, we note that we further have
			\begin{align*}
				&\bE_i\big[ \abs{Y_\tip}^8 \big]^{1/4} \bE_i\big[\abs{Y_\tip}^{2p}\big]^{1/4} 
					\le \bE_i\big[\abs{Y_\tip}^{2p}\big]^{1/p+1/4} 
					\le 1 + \bE_i\big[\abs{Y_\tip}^{2p}\big]  
				\qquad \text{and} \\
				 &\bE_i\big[ \abs{Y_\tip}^4 \big]^{1/2} \bE_i\big[\abs{Y_\tip}^p\big]^{1/2}
				 	\le \bE_i\big[\abs{Y_\tip}^p\big]^{2/p+1/2} 
				 	\le 1 + \bE_i\big[\abs{Y_\tip}^p\big].
			\end{align*}
		Finally, note that since $p \ge 12(m-1)$,
			\begin{align*}
				\big( 1 + \bE_i\big[\abs{Y_\tip}^{4(m-1)}\big] \big)^\half \big( 1 + \bE_i\big[\abs{Y_\tip}^{2p}\big] \big)	
					&
					\le \big( 1 + \bE_i\big[\abs{Y_\tip}^{4(m-1)}\big] \big) \big( 1 + \bE_i\big[\abs{Y_\tip}^{2p}\big] \big)	
					\\
					&\le C \big( 1 + \bE_i\big[\abs{Y_\tip}^{4p}\big] \big),
			\end{align*}
		and thus
				$\Abs{ \tau^{Y,t}_i }^2 \le \cP(Y_\tip) h^3$,
		which completes the proof.
	\end{proof}

\begin{remark}[On higher order schemes]
Notice that, formally, writing the steps of the scheme as $\Phi_i = \Phi^{0,h}_i = \Phi^0_i \circ T^h$, and denoting by $\Psi_i$ the BSDE dynamics over $[\ti,\tip]$, this proof  works essentially by decomposing $\Psi_i - \Phi_i = \Psi_i - \Phi^0_i + \Phi^0_i - \Phi_i = \Psi_i - \Phi^0_i + \Phi^0_i \circ (Id-T^h)$. Since the $p$'s can be chosen arbitrarily large in the proof above, the estimates related to the second term, i.e. to the projection $T^h$, can be of arbitrarily-high order of $h$. 
Meanwhile, the first term is the error of the non-projected scheme.
Thus, should the Full-Projection Scheme be built from a higher-order scheme $(\Phi^{+}_i)$ (see for instance \cite{Chassagneux2014} and \cite{ChassagneuxCrisan2014}), rather than the explicit BTZ scheme, we expect the resulting scheme $(\Phi^{+,h}_i=\Phi^{+}_i \circ T^h)$ to be of the same high-order.
\end{remark}
We can now combine the local error estimate to obtain the total time-discretization error.
	\begin{proposition}		
	\label{proposition---error.time.discretization.global.for.RV}
		There exist a constant $C \ge 0$ such that
			\begin{align*}
				\cT^Y(\pi) 
				= \sum_{i=0}^{N-1} \frac{1}{h} \bE\big[\,\abs{\tau^Y_i}^2\,\big] \le C h
				\qquad \text{and} \qquad 
				\cT^Z(\pi) 
				= \sum_{i=0}^{N-1} \bE\big[\,\abs{\tau^Z_i}^2 h\,\big] \le C h.
			\end{align*}
	\end{proposition}

	\begin{proof}
		First, notice that in the previous proof, immediately after squaring and before using any Cauchy--Schwarz inequality, we can take the expectation.
		So all conditional expectations can be replaced by expectations.
		Second, it is a standard bound for the solution to the BSDE (see \cite{LionnetReisSzpruch2015}, section 2) that for all $q \ge 1$, there exists $C_q \ge 0$ such that
		$\bE\big[ \sup_{0\le t \le T} \abs{Y_t}^q \big] \le C_q$.
		We therefore see that for all $i$, $\cP\big(Y_\tip\big) \le C$ and $\cP\big(\sup_{\ti \le u \le \tip} \abs{Y_u}\big) \le C$, 
		where $\cP$ is the analogue of $\cP_i$ with expectations instead of conditional expectations.
		For the $\tau^Z_i$'s this yields
			\begin{align*}
				\sum_{i=0}^{N-1} \bE[\abs{\tau^Z_i}^2] \le C \bE\bigg[ \int_0^T \abs{Z_u}^2 du \bigg] + C \bE\bigg[ \int_0^T \abs{f(Y_u,Z_u)}^2 \ud u \bigg] + C,
			\end{align*}
		and the result follows from standard bounds for the solution to the BSDE .
		For the $\tau^Y_i$'s this yields
			\begin{align*}
				\sum_{i=0}^{N-1} \frac{1}{h} \bE[\abs{\tau^Y_i}^2] 
					&\le C h \sum_{i=0}^{N-1} \bE\big[\sup_{u\in[\ti,\tip]} \abs{Y_u-Y_\tip}^4\big]^\half 
									+ C \sum_{i=0}^{N-1} \bE\bigg[\int_\ti^\tip \abs{Z_u-\wb Z_\ti}^2\bigg] 							
									+ C \cT^Z(\pi)
									+ C T h,
			\end{align*}
		and the results then follows from the path-regularity theorems (see \cite{LionnetReisSzpruch2015}*{Section 3.4}) and the bound for $\cT^Z(\pi)$.
	\end{proof}

\subsection{Convergence of the scheme} 
\label{sec:proofofConvTheorem} 

	\begin{proof}[Proof of Theorem \ref{proposition---MainTheorem}]
		Given all the cumputations in the previous section, we decompose the error at time $\ti$ in the usual way of local errors plus the propagation of the error existing at time $\tip$:
			\begin{align*}
				& Y_\ti - Y_i = Y_\ti - \wh Y_i + \wh Y_i - Y_i = \tau^Y_i + \Phi^Y_i(Y_\tip) - \Phi^Y_i(Y_\ip)
				\qquad \text{and} \\
				& \wb Z_\ti - Z_i = \wb Z_\ti - \wh Z_i + \wh Z_i - Z_i = \tau^Z_i + \Phi^Z_i(Y_\tip) - \Phi^Z_i(Y_\ip)				
			\end{align*}
		where the random variables $(\wh Y_i,\wh Z_i)$ are defined in \eqref{equation---definition.of.hatYi.hatZi}.
		Given the stability of the scheme stated in Proposition \ref{proposition---stability.of.the.scheme}, 
		we can square the above and ``iterate'' (literally if we are only interested in $\mathrm{ERR}^Y(\pi)$, using in fact Lemma A.3 in \cite{LionnetReisSzpruch2016} if not) to obtain that
			\begin{align*}
				\abs{Y_\ti-Y_i}^2 + \bE_i\bigg[ \sum_{j=i}^{N-1} \abs{\wb Z_\tj-Z_j}^2 h \bigg] 
						\le e^{c (T-\ti)} \bE_i\bigg[ \abs{g(X_T)-g(X_N)}^2 + C \ \sum_{j=i}^{N-1}  \frac{1}{h} \abs{\tau^Y_j}^2 + \abs{\tau^Z_j}^2 h \bigg] .
			\end{align*}
		It then follows that
			\begin{align*}
				\mathrm{ERR}^Y(\pi)^2 + \mathrm{ERR}^Z(\pi)^2 
						\le C \bigg\{ \bE[\abs{g(X_T)-g(X_N)}^2] + \sum_{i=0}^{N-1}  \frac{1}{h} \bE[\abs{\tau^Y_i}^2] + \bE[\abs{\tau^Z_i}^2] h \bigg\} .
			\end{align*}
		As $g$ is Lipschitz, the explicit Euler scheme for $X$ has squared error bounded by $C h$, 
		and we have estimated the sum of local time-discretization errors in Proposition \ref{proposition---error.time.discretization.global.for.RV},
		the proof is complete.		
	\end{proof}

\section{Analysis of the full scheme's numerical error} 		

We now go back to Section \ref{section---numerical.scheme.and.results} and prove Theorem \ref{proposition---weak.error.estimate.global}. At this point we invite the reader to recall Section \ref{sec:BasicApproxofExpectations} and \ref{sec:mainresultNumerics} where the introduced a simpler version of the spatial discretization procedure. In the remainder of this section, we will work with the general version.

Recall the PDE \eqref{eq:PDEforErrorEstimate}, we write $y_i = y(\ti,\cdot)$ and $z_i = z(\ti,\cdot)$. As in \cite{Chassagneux2014}, we compare the process $(\wb Y_i)$ in \eqref{equation---reference-YZ.scheme--fully.discretized} with the process $(\cY_i)$ where $\cY_i = y_i(\wb X_i)$ (from the spatial discretization). 
In order to write the decomposition of the error and analyse it, we introduce a family of operators.

Let $\Phi_i : u \mapsto (\Phi^Y_i(u),\Phi^Z_i(u))$ be defined, for a measurable map $u$, as 
	\begin{align*}
		&\left\{\begin{aligned}
			\Phi^Y_i(u)(x) = \wh Y_i^{\ti,x} &:= \bE^{\ti,x}\Big[ T^h(u(X_\ip)) + f\big(T^h(u(X_\ip)),\Phi^Z_i(u)(x)\big) h \Big]				\\
			\Phi^Z_i(u)(x) = \wh Z_i^{\ti,x} &:= \bE^{\ti,x}\big[  T^h(u(X_\ip)) \Hip^* \big], 
		\end{aligned}\right. 
	\end{align*}
where $X_\ip = x + b(\ti,x) h + \sigma(\ti,x)\Delta W_\tip$ is the output of the Euler scheme \eqref{equation---reference--X.Euler.scheme} started from $x$ at $\ti$.
We define for all $i = 0, \ldots, N-1$, $(\wh y_i,\wh z_i):=\Phi_i(y_\ip)$.

Similarly, let $\wb \Phi_i : u \mapsto (\wb\Phi^Y_i(u),\wb\Phi^Z_i(u))$ be defined by 
	\begin{align*}
		&\left\{\begin{aligned}
			\wb\Phi^Y_i(u)(x) = \wb Y_i^{\ti,x} 
			&:= \bE^{\ti,x}\Big[ T^h(u(\wb X_\ip)) + f\big(T^h(u(\wb X_\ip)), \wb\Phi^Z_i(u)(x) \big) h \Big]		
			\\
			\wb\Phi^Z_i(u)(x) = \wb Z_i^{\ti,x} 
			&:= \bE^{\ti,x}\big[  T^h(u(\wb X_\ip)) \wbHip^* \big], 
		\end{aligned}\right. 
	\end{align*}
where $\wb X_\ip = \Pi\big[ \wt X_\ip \big] = \Pi \big[ x + b(\ti,x) h + \sigma(\ti,x)\wb{\Delta W}_\tip \big]$ is the output of the quantized schemed \eqref{equation---reference--X.scheme--fully.discretized} started from $x$ at $\ti$.
We let, for all $i = 0, \ldots, N-1$, $(\rh y_i, \rh z_i) = \wb \Phi_i(y_\ip)$.
Expressing the scheme \eqref{equation---reference-YZ.scheme--fully.discretized} in terms of the $\wb \Phi_i$ we thus see that $\wb y_i$ and $\wb z_i$ are defined recursively: for all $i$, $(\wb y_i,\wb z_i) = \wb \Phi_i(\wb y_\ip)$.

Finally, let $\Psi_i : u \mapsto (\Psi^Y_i(u),\Psi^Z_i(u))$ where, for all $x$, $(\Psi^Y_i(u)(x),\Psi^Z_i(u)(x)) = (Y^{\ti,x,u}_\ti,\wb Z^{\ti,x,u}_\ti)$, where $\wb Z^{\ti,x,u}_\ti = h^{-1}\bE_i[\int_\ti^\tip Z^{\ti,x,u}_t dt]$, $(Y^{\ti,x,u}_t,Z^{\ti,x,u}_t)_{t \in [\ti,\tip]}$ is the solution over $[\ti,\tip]$ to the BSDE 
with driver $f$ and terminal condition $u(X^{\ti,x}_\tip)$ at time $\tip$, and $(X^{\ti,x}_t)_{t \in [\ti,\tip]}$ is the solution to the SDE \eqref{equation---canonicSDE} starting from $x$ at $\ti$.
Naturally, we have for all $i$, $y_i = \Psi^Y_i(y_\ip)$.

Note that we are overloading the operators $\Psi_i$ and $\Phi_i$. Nonetheless, there is no ambiguity, since the previously introduced versions take random variables as inputs  whereas the ones above take functions as inputs.

\subsection{Proof of the total error estimate (Theorem \ref{proposition---weak.error.estimate.global})} 
\label{sec:proofofmainnumericstheorem}

		The strategy is similar to the proof of Theorem \ref{proposition---MainTheorem}.
		We decompose the error at time $\ti$ as a local error plus the propagation of the error existing at time $\tip$.
		We thus write, for any $x=\wb X^{0,x_0}_i$, 
			\begin{align*}
				y_i(x) - \wb y_i(x)
					&= \big( y_i(x) - \wh y_i(x) \big) + \big( \wh y_i(x) - \rh y_i(x) \big) + \big( \rh y_i(x) - \wb y_i(x) \big)	
					\\
					&= \big[ \Psi^Y_i(y_\ip) - \Phi^Y_i(y_\ip) + \Phi^Y_i(y_\ip) - \wb\Phi^Y_i(y_\ip) + \wb\Phi^Y_i(y_\ip) - \wb\Phi^Y_i(\wb y_\ip) \big](x)
					\\
					&= \tau^Y_i(x) + \epsilon^Y_i(x) + \rho^Y_i(x).
			\end{align*}
		The difference with the proof of Theorem \ref{proposition---MainTheorem} is that the local error now has two sources: the time discretization, $\tau^Y_i$, 
		and the approximation of the expectation, $\epsilon^Y_i$.
		Again, the key for iterating is a stability estimate for the full scheme. This then yields the form of the global error estimate.

		\paragraph*{}\emph{Stability of the scheme.}
			Since $\wbHip$ satisfies the assumptions (a) and (b) satisfied by $\Hip$, 
			the proof of Proposition \ref{proposition---stability.of.the.scheme} carries over to prove that the scheme $\wb \Phi_i$ is stable.
			So we have 
				\begin{align*}
					\abs{\rho^Y_i(x)}^2 
					= \abs{ \wb\Phi^Y_i(y_\ip)(x) - \wb\Phi^Y_i(\wb y_\ip)(x) }^2 	
					\le e^{c h} \ \bE^{\ti,x}\big[ \abs{ (y_\ip - \wb y_\ip) (\wb X^{\ti,x}_\ip) }^2 \big] .
				\end{align*}

		\paragraph*{}\emph{Global error bound.}
			Taking the square in the decomposition above and iterating we obtain
				\begin{align*}
					\abs{y_i(x) - \wb y_i(x)}^2
						&\le e^{c(T-\ti)} \bE^{\ti,x}\Big[ \abs{(y_N-\wb y_N)(\wb X^{\ti,x}_N)}^2 
						     + C \sum_{j=i}^{N-1} \frac{1}{h} \abs{(\tau^Y_j+\epsilon^Y_j)(\wb X^{\ti,x}_j)}^2 \Big] .
				\end{align*}
			Consequently, we have for $i=0$, using $(a+b)^2 \le 2(a^2+b^2)$, $y_N=\wb y_N=g$, and with $\wb X_i = \wb X^{0,x_0}_i$,
				\begin{align} 	\label{equation---error.estimate.full.scheme.general.form}
					\abs{Y_0 - \wb Y_0}^2 = \abs{y_0(x_0) - \wb y_0(x_0)}^2
						&\le 
							C \ \sum_{i=0}^{N-1} \frac{1}{h} 
							         \bE^{0,x_0}\big[ \abs{\tau^Y_i(\wb X_i)}^2 + \abs{\epsilon^Y_i(\wb X_i)}^2 \big] .
				\end{align}
			Thus, proving Proposition \ref{proposition---weak.error.estimate.global} for the full scheme reduces to proving the following two lemmas,
			which allow to bound the RHS of the fundamental estimate \eqref{equation---error.estimate.full.scheme.general.form},
			since for any $p \ge 1$ there exists $C_p$ such that $\sup_i \bE^{0,x_0}[\abs{\wb X_i}^p] \le C_p$.
			Similarly to Lemma \ref{proposition---error.time.discretization.local.for.RV} and Proposition \ref{proposition---error.time.discretization.global.for.RV}, 
			we denote by $P(x)$ the quantity $C(1+\abs{x}^p)$, for constants $C$ and $p$ not depending on $\pi$ and which may vary from line to line (this notation is linked to that used in Section \ref{sec:LocalTime-DiscretizationError}).

	\begin{lemma}	\label{lemma---weak.local.error.time-discretization}
		For any $i\in\{0,\cdots,N-1\}$ and $x\in \Gamma$ we have
		$\abs{\tau^Y_i(x)} \le P(x) h^2$. 
	\end{lemma}

	\begin{lemma}	\label{lemma---weak.local.error.approximated-expectations}
		For any $i\in\{0,\cdots,N-1\}$ and $x\in \Gamma_i$, if $q \ge 3$, we have
				$\abs{\epsilon^Y_i(x)} \le P(x) h^2 + C \eta$.
	\end{lemma}

		Note that there is no term $\tau^Z+\epsilon^Z$ in \eqref{equation---error.estimate.full.scheme.general.form} since we only iterated the $Y$-estimate.
		However these terms still appear in the estimation of $\tau^Y+\epsilon^Y$. We start with those in each case.

	\begin{proof}[Proof of lemma \ref{lemma---weak.local.error.time-discretization}]
	
			We first consider $\tau^Z_i$ defined for any $i$ and $x\in \Gamma_i$ by
				\begin{align*}
					\tau^Z_i(x) := \Psi^Z_i(y_\ip)(x) - \Phi^Z_i(y_\ip)(x)
							= \wb Z^{\ti,x}_\ti - \wh z_i(x).
				\end{align*}
			We note that
				\begin{align*}
					\wb Z^{\ti,x}_\ti - \wh z_i(x) 
						&= \bEix\bigg[ \int_\ti^\tip z(u,X_u) dW_u \ \frac{\Delta W_\tip}{h} \bigg] - \bEix\Big[ T^h\big(y_\ip(X_\ip)\big) \Hip \Big]	\\
						&= \bEix\bigg[ \int_\ti^\tip z(u,X_u) dW_u \Big( \frac{\Delta W_\tip}{h} - \Hip \Big) \bigg]														\\
									& \qquad + \bEix\bigg[ \Big( y_\ip(X_\tip) + \int_\ti^\tip f\big(y(u,X_u),z(u,X_u)\big)du - y_\ip(X_\tip) \Big) \ \Hip \bigg]	 				\\
									& \qquad + \bEix\bigg[ \Big( y_\ip(X_\tip) - T^h\big( y_\ip(X_\tip) \big) \Big) \ \Hip \bigg]	 				\\
									& \qquad + \bEix\bigg[ \Big( T^h\big( y_\ip(X_\tip)\big) - T^h\big( y_\ip(X_\ip)\big) \Big) \ \Hip \bigg]		\\
						&= \alpha^{Z,1}_i(x) + \alpha^{Z,2}_i(x) + \alpha^{Z,3}_i(x) + \alpha^{Z,4}_i(x).
				\end{align*}
			For the first term, we use the Cauchy--Schwartz inequality to write
				\begin{align*}
					\abs{\alpha^{Z,1}_i(x)} \le \bEix\bigg[ \int_\ti^\tip \abs{z(u,X_u)}^2 du \bigg]^{1/2} \bEix\bigg[ \Abs{ \frac{\Delta W_\tip}{h} - \Hip }^2 \bigg]^{1/2}
						&\le P(x) h^{1/2} \ C h^{1/2} = P(x) h,
				\end{align*}
			given the growth assumption on $z$, the integrability of $(X_u)_{\ti \le u \le \tip}$ and assumption (c) on $\Hip$.			
			For the second term, we use It\^o's formula and the martingale property of $\Hip$, to have
				\begin{align*}
					\alpha^{Z,2}_i(x) = 0 + \bEix\bigg[ \int_\ti^\tip \int_\ti^u L\big(f(y,z)\big)(t,X_t) dt + (\partial_x f(y,z) \, \sigma)(t,X_t) dW_t \, du \ \ \Hip \bigg],
				\end{align*}
			where $f(y,z)$ here is $(t,x) \mapsto f\big(y(t,x),z(t,x)\big)$, and we denote by $L(\varphi)$, for any smooth function $\varphi$, the function $(t,x) \mapsto L(\varphi)(t,x)$ defined by
				\begin{align*}
					L(\varphi)(t,x) = \partial_t \varphi(t,x) + \half \trace \big[ \partial_{xx} \varphi \ \sigma \sigma^* \big](t,x)+ \big(\partial_x \varphi \ b\big)(t,x) . 
				\end{align*}
			Using again the Cauchy--Schwartz inequality, the growth of $y$, $z$ and the coefficients, the integrability $(X_u)_{\ti \le u \le \tip}$, and assumption (b) on $\Hip$, 
			we thus find that $\abs{\alpha^{Z,2}_i(x)} \le P(x) h$.
			For the third term, using the Markov inequality in the same way as in the proof of lemma \ref{proposition---error.time.discretization.local.for.RV}, we can obtain arbitrarily-high orders of $h$,
			in particular we can obtain $\abs{\alpha^{Z,3}_i(x)} \le P(x) h$.
			Finally, for the fourth term, denoting by $\varphi : (x,\delta w) \mapsto T^h\big( y_\ip(x) \big) \ T^{r^h}(\delta w) / h$, we have
				\begin{align*}
					\alpha^{Z,4}_i(x) = \bEix\Big[ \varphi(X_\tip,\Delta W_\tip) \Big] - \bEix\Big[ \varphi(X_\ip,\Delta W_\tip) \Big].
				\end{align*}
			Since $\varphi$ is smooth with polynomially-growing derivatives and the Euler scheme has local weak error of order $h^2$, cf \cite{KloedenPlaten1992}, 
			this term can be bounded by $P(x)h$.

			In the end, we have obtained $\abs{\tau^Z_i(x)} = \abs{\wb Z^{\ti,x}_\ti - \wh z_i(x)} \le P(x) h$.

			\paragraph*{}			
			We now consider $\tau^Y_i(x)$. We decompose it as
				\begin{align*}
					\tau^Y_i(x) 
						& = \Psi^Y_i(y_\ip)(x) - \Phi^Y_i(y_\ip)(x)
						= y_i(x) - \wt y_i(x) + \wt y_i(x) - \wh y_i(x),
				\end{align*}
			where for convenience we introduce 
$\wt y_i(x) = \bEix\Big[ y_\ip(X_\tip) + f\big( y_\ip(X_\tip), \wb Z^{\ti,x}_\ti \big) h \Big]$.
			
			We start by estimating the first term of the decomposition. We note that, using It\^o's formula,
				\begin{align*}
					y_i(x) &= \bEix\bigg[ y_\ip(X_\tip) + \int_\ti^\tip f(y,z)(u,X_u) du \bigg]																																			\\
						&= \bEix\bigg[ y_\ip(X_\tip) + f\big(y_\ip(X_\tip),z_\ip(X_\tip)\big) h \bigg] + \bEix\bigg[ -\int_\ti^\tip \int_u^\tip L\big(f(y,z)\big)(t,X_t) dt du \bigg]	- 0		\\
						&= \bEix\bigg[ y_\ip(X_\tip) + f\big(y_\ip(X_\tip), \wb Z^{\ti,x}_\ti \big) h \bigg] 																														\\
							&\qquad + \bEix\bigg[ f\big(y_\ip(X_\tip), z_i(x) \big) - f\big(y_\ip(X_\tip), \wb Z^{\ti,x}_\ti \big) \bigg] h																				\\
							&\qquad + \bEix\bigg[ f\big(y_\ip(X_\tip), z_\ip(X_\tip)\big) - f\big(y_\ip(X_\tip), z_i(x) \big) \bigg] h																					\\
							&\qquad \qquad - \bEix\bigg[ \int_\ti^\tip \int_u^\tip L\big(f(y,z)\big)(t,X_t) dt du \bigg]																										\\
						&= \wt y_i(x) + \alpha_i^{Y,1}(x) + \alpha_i^{Y,2}(x) + \alpha^{Y,3}_i(x).
				\end{align*}
			Using the growth of the functions and the integrability of the forward process, the fourth term can be bounded above by $P(x) h^2$.
			For the second term, using \hfreg we obtain
				\begin{align*}
					\abs{\alpha_i^{Y,1}(x)} \le L_z \abs{z_i(x) - \wb Z^{\ti,x}_\ti} h \le P(x) h^2 .
				\end{align*}
			Indeed, using It\^o's formula we can write
				\begin{align*}
					\wb Z^{\ti,x}_\ti = \bEix\bigg[ \frac1h \int_\ti^\tip z(u,X_u) du \bigg]
						= z_i(x) + \bEix\bigg[ \frac1h \int_\ti^\tip \int_\ti^u L(z)(t,X_t) dt du \bigg]
				\end{align*}
			and the integral term can be bounded above by $P(x) h$.
			For the third term, we write, using again It\^o's formula,
				\begin{align*}
					\alpha_i^{Y,2}(x) = \bEix\bigg[ 0 + \int_\ti^\tip \big[ L\big(f(y,z)\big) - L\big(f(y,z_i(x))\big) \big](u,X_u) du  + 0 \bigg] h,
				\end{align*}
			where naturally $f\big(y,z_i(x)\big) : (t,X) \mapsto f\big(y(t,X),z_i(x)\big)$. We thus estimate that $\abs{\alpha_i^{Y,2}(x)} \le P(x) h^2$.			
			To conclude, we have proved that 
					$\abs{y_i(x) - \wt y_i(x)} \le P(x) h^2$.

			We now estimate the second term of the decomposition, 
				\begin{align*}
					\wt y_i(x) - \wh y_i(x)	
						&= \bEix\Big[ y_\ip(X_\tip) + f\big( y_\ip(X_\tip), \wb Z^{\ti,x}_\ti \big) h \Big] 																											\\
									&\hspace{5cm} - \bEix\Big[T^h(y_\ip(X_\ip)) + f\big( T^h(y_\ip(X_\ip)), \wh z_i(x) \big) h \Big]																		\\
						&= \bEix\Big[ y_\ip(X_\tip) + f\big( y_\ip(X_\tip), \wb Z^{\ti,x}_\ti \big) h \Big] 																											\\
									&\hspace{3cm} - \bEix\Big[T^h(y_\ip(X_\tip)) + f\big( T^h(y_\ip(X_\tip)), \wb Z^{\ti,x}_\ti \big) h \Big]															\\
							&\quad + \bEix\Big[T^h(y_\ip(X_\tip)) + f\big( T^h(y_\ip(X_\tip)), \wb Z^{\ti,x}_\ti \big) h \Big] 																				\\
									&\hspace{3cm} - \bEix\Big[T^h(y_\ip(X_\ip)) + f\big( T^h(y_\ip(X_\ip)), \wb Z^{\ti,x}_\ti \big) h\Big] 																\\
							& \quad + \bEix\Big[ f\big( T^h(y_\ip(X_\ip)), \wb Z^{\ti,x}_\ti \big) - f\big( T^h(y_\ip(X_\ip)), \wh z_i(x) \big) \Big] h											\\
						&= \beta^{Y,1}_i(x) +\beta^{Y,2}_i(x) + \beta^{Y,3}_i(x).
				\end{align*}
			The first term, $\beta^{Y,1}_i(x)$, can be estimated using the Markov inequality, and be bounded above by $P(x)h^2$.
			{For the second term, 
			since the function $\varphi : X \mapsto T^h(y_\ip(X)) + f\big( T^h(y_\ip(X)), \wb Z^{\ti,x}_\ti \big) h$ is smooth with polynomial-growth derivatives
			and the forward Euler scheme has a local weak error of order $h^2$, $\abs{\beta^{Y,2}_i(x)}$ is bounded above by $P(x) h^2$.} 
			For the third term, we have the bound
				\begin{align*}
					\abs{\beta^{Y,3}_i(x)} \le L_z \abs{\wb Z^{\ti,x}_\ti - \wh z_i(x)} h = L_z\abs{\tau^Z_i(x)}h\le P(x) h^2,
				\end{align*}
			by the previous estimation of $\abs{\tau^Z_i(x)}$.
			So we have 
					$\abs{\wt y_i(x) - \wh y_i(x)} \le P(x) h^2$.

			Collecting all the estimates yields easily that $\abs{\tau^Y_i(x)} \le P(x) h^2$, which closes the proof.
	\end{proof}

	\begin{proof}[Proof of lemma \ref{lemma---weak.local.error.approximated-expectations}]

			We first consider $\epsilon^Z_i$ (defined below). For any $x$, $i$, 
				\begin{align*}
					\epsilon^Z_i(x)
						:& = \Phi^Z_i(y_\ip)(x) - \wb \Phi^Z_i(y_\ip)(x)																																																	\\
						& = \bE^{\ti,x}\big[  T^h(y_\ip(X_\ip)) \Hip^* \big] - \bE^{\ti,x}\big[  T^h(y_\ip(\wb X_\ip)) \wbHip^* \big]																									\\
						& = \bE^{\ti,x}\Big[  T^h(y_\ip(X_\ip)) \ T^{r^h}(\Delta W_\tip)^* \Big] h^{-1} - \bE^{\ti,x}\Big[  T^h(y_\ip(\wt X_\ip)) \ T^{r^h}(\wb{\Delta W}_\tip)^* \Big] h^{-1}					\\
								&\qquad \hspace{1.35cm} + \bE^{\ti,x}\Big[  T^h(y_\ip(\wt X_\ip)) \ \wbHip^*\Big] - \bE^{\ti,x}\Big[  T^h(y_\ip(\wb X_\ip)) \ \wbHip^*\Big]											\\
						& = \epsilon^{Z,g}_i(x) + \epsilon^{Z,\eta}_i(x).
				\end{align*}
			In the equation above,
				\begin{align*}
					X_\ip &= X_\ip(\Delta W_\tip) = x + b(\ti,x) h + \sigma(\ti,x) \Delta W_\tip,																	\\
					\wt X_\ip &= \wt X_\ip(\wb{\Delta W}_\tip) = x + b(\ti,x) h + \sigma(\ti,x) \wb{\Delta W}_\tip, \quad \text{and}\quad
					\wb X_\ip 
					= \Pi(\wt X_\ip).
				\end{align*}
			By assumption $y_\ip$ and $T^h$ are smooth with bounded derivatives. 
			We can therefore write the first term in the error decomposition of $\epsilon^Z_i(x)$ as 
				\begin{align*}
					\epsilon^{Z,g}_i(x) = \bE^{\ti,x}\big[ \varphi(\Delta W_\tip) \big] h^{-1} - \bE^{\ti,x}\big[ \varphi(\wb{\Delta W}_\tip) \big] h^{-1},
				\end{align*}
			with a smooth function $\varphi$.
			We can thus expand both expectations as
				\begin{align*}
					\bE^{\ti,x}\big[ \varphi(\Delta W_\tip) \big]
						= \sum_{k=0}^{q} \frac{\varphi^{(k)}(0)}{k!} \bE^{\ti,x}\big[ (\Delta W_\tip)^{\otimes q} \big] + R^\varphi h^{(q+1)/2} ,
				\end{align*}
			where the $\varphi^{(k)}(0)$ and $R^\varphi$ have polynomial growth in $x$.
			Since $\wb{\Delta W}_\tip$ matches the moments of $\Delta W_\tip \sim \cN(0,h)$ up to order $q$, we have 
					$\abs{\epsilon^{Z,g}_i(x)} \le P(x) h^{(q+1)/2} h^{-1} = P(x) h^{(q-1)/2}$. 
			Meanwhile, since $T^h$ and $y_\ip$ are Lipschitz, and $\abs{\wt X_\ip-\wb X_\ip}=\abs{\wt X_\ip-\Pi(\wt X_\ip)} \le \eta$,
				\begin{align*}
					\abs{\epsilon^{Z,\eta}_i(x)} 
						\le C \bE^{\ti,x}\Big[ \abs{\wt X_\ip-\wb X_\ip} \abs{\wbHip} \Big]
						& \le C \eta h^{-1/2},
				\end{align*}
			Consequently,
				\begin{align}
				\label{eq:EtaAppears1}
					\abs{\epsilon^Z_i(x)} 
						\le C h^{(q-1)/2} + C \eta h^{-1/2}.
				\end{align}

			Next, we consider $\epsilon^Y_i$. For any $x$ and $i$, we decompose it as
				\begin{align*}
					\epsilon^Y_i(x) 
						&= \Phi^Y_i(y_\ip)(x) - \wb \Phi^Y_i(y_\ip)(x)																																					\\
						& =  \bE^{\ti,x}\Big[ T^h(y_\ip(X_\ip)) + f\big(T^h(y_\ip(X_\ip)),\wh z_i(x)\big) h \Big] 																					\\
								&\qquad - \bE^{\ti,x}\Big[ T^h(y_\ip(\wb X_\ip)) + f\big(T^h(y_\ip(\wb X_\ip)),\wb z_i(x)\big) h \Big]													\\
						& =  \bE^{\ti,x}\Big[f\big(T^h(y_\ip(X_\ip)),\wh z_i(x)\big) - f\big(T^h(y_\ip(X_\ip)),\wb z_i(x)\big) \Big] h														\\
								&\quad + \bE^{\ti,x}\Big[ T^h(y_\ip(X_\ip)) + f\big(T^h(y_\ip(X_\ip)),\wb z_i(x)\big) h \Big] 																	\\
								&\hspace{2cm} - \bE^{\ti,x}\Big[ T^h(y_\ip(\wt X_\ip)) + f\big(T^h(y_\ip(\wt X_\ip)),\wb z_i(x)\big) h \Big] 											\\
								&\quad + \bE^{\ti,x}\Big[ T^h(y_\ip(\wt X_\ip)) + f\big(T^h(y_\ip(\wt X_\ip)),\wb z_i(x)\big) h \Big]														\\
								&\hspace{2cm} - \bE^{\ti,x}\Big[ T^h(y_\ip(\wb X_\ip)) + f\big(T^h(y_\ip(\wb X_\ip)),\wb z_i(x)\big) h \Big].
				\end{align*}
			Using that $z\mapsto f(\cdot,z)$ is Lipschitz and $\abs{\wt X_\ip-\wb X_\ip}=\abs{\wt X_\ip-\Pi(\wt X_\ip)} \le \eta$ plus the Lipschitz property of $T^h$ and $y_\ip$ we have
				\begin{align*}
					\abs{\epsilon^Y_i(x)}
						&\le  L_z \abs{\epsilon^Z_i(x)} h	
										+ \Big|\bE^{\ti,x}\Big[ T^h(y_\ip(X_\ip)) + f\big(T^h(y_\ip(X_\ip)),\wb z_i(x)\big) h \Big]
										\\
								&\hspace{3.4cm} - \bE^{\ti,x}\Big[ T^h(y_\ip(\wt X_\ip)) + f\big(T^h(y_\ip(\wt X_\ip)),\wb z_i(x)\big) h \Big] \Big| 															
								+ C \eta + C \eta h.
				\end{align*}
			Arguing as above, since $\wb{\Delta W}_\tip$ has the same moments as $\Delta W_\tip$ up to order $q$, we obtain
				\begin{align}
				\nonumber
					\abs{\epsilon^Y_i(x)} 
						&\le P(x) \abs{\epsilon^Z_i(x)} h + P(x) h^{(q+1)/2} + C \eta + C \eta h
						\\
						&\le P(x) \big[ h^{(q-1)/2} + \eta h^{-1/2} \big] h + P(x) h^{(q+1)/2} + C \eta + C \eta h
						\label{eq:EtaAppears2}
						\le P(x) h^{(q+1)/2} + C \eta.
				\end{align}
			using the previously obtained estimate on $\epsilon^Z_i$
			and using $h \le h_{\text{max}} \le 1$.
			Since $q \ge 3$, this completes the proof.
	\end{proof}

When the dynamic of the forward process is that of a linear Brownian motion ($b$ and $\sigma$ are constants), one sees that the forward dynamics with merely the quantized increments and no projection, defined by $\wb X_0=x_0$ and 
	\begin{align*}
		\wb X_\ip = \wb X_i + b \, h + \sigma \, \wb{\Delta W}_\ip
	\end{align*}
has a distribution supported by a recombining tree. In that situation, there is no need to define a grid $\Gamma$ with mesh $\eta$ (which should depend on $h$) and no need for the projection $\Pi$ at each step. 
Inspecting the proof above, we see that at each step $\wt X_\ip = \wb X_\ip$, and hence all the terms containing $\eta$ disappear from \eqref{eq:EtaAppears1} and \eqref{eq:EtaAppears2}. 
We thus have the following better error bound.
	\begin{proposition} 	
	\label{proposition---weak.error.estimate.global.good.case}
		Assume that $b$ and $\sigma$ are constant.
		There exists a $C \ge 0$ such that  
				$\abs{Y_0 - \wb Y_0} \le C h$.
	\end{proposition}


\section{Conclusions}

In this work, we have introduced the Full-Projection explicit time-discretization scheme for the approximation of FBSDEs/nonlinear PDEs with non-globally Lipschitz drivers of polynomial growth. This scheme covers a class of nonlinearities found in important applications/PDE models, e.g. Allen--Cahn equation, the FitzHugh--Nagumo equations or the Fisher--KPP equation, to mention a few.

We provided a complete study of all the errors related to the fully implementation of the FBSDE based algorithm. We highlighted several important numerical stability features of the scheme, in particular, its ability to preserve the original PDE's dissipative nature (when $M_y<0$ in \hfmon).

\bibliographystyle{alpha} 

\bibliography{FullBSDEpolyYtamed-FPSClean} 
\end{document}